\documentclass[reqno]{amsart}
\usepackage{amsmath}
\usepackage{amssymb}
\usepackage{amsthm}
\usepackage{amsfonts,dsfont}
\usepackage{mathtools}
\usepackage{bm}
\usepackage[normalem]{ulem}
\usepackage{xcolor}
\usepackage{fancyhdr}
\usepackage{hyperref}
\usepackage{mdframed}
\usepackage{mathrsfs}
\usepackage{comment}
\usepackage{graphicx,lipsum,wrapfig}
\usepackage[font={tiny}]{caption}
\usepackage[font={scriptsize}]{caption}
\usepackage{tikz}
\usepackage{pgfplots}
\usepackage{soul}

\theoremstyle{plain}
\newtheorem{thm}{Theorem}[section]
\newtheorem{corollary}[thm]{Corollary}
\newtheorem{lemma}[thm]{Lemma}
\newtheorem{prop}[thm]{Proposition}
\newtheorem*{theorem*}{Theorem}
\theoremstyle{definition}
\newtheorem{rem}[thm]{Remark}

\newtheorem{assumption}{Assumption}
\numberwithin{equation}{section}

\newcommand{\mytilde}{\raise.17ex\hbox{$\scriptstyle\mathtt{\sim}$}}

\newcommand{\Rat}[1]{\ensuremath{\mathbb{R}^{#1}}}

\renewcommand{\P}{\mathcal{P}}

\newcommand{\bE}{\ensuremath{\mathbb{E}}}
\newcommand{\eps}{\ensuremath{\varepsilon}}

\newcommand{\longthmtitle}[1]{\mbox{}{\bf \textit{(#1).}}}

\allowdisplaybreaks

\title[Structured ambiguity sets]{Structured ambiguity sets for distributionally robust optimization}

\begin{document}

\author{Lotfi M. Chaouach}

\author{Tom Oomen}

\author{Dimitris Boskos}

\thanks{All the authors are with the Delft Center for Systems and Control, Faculty of Mechanical, Maritime and Materials Engineering, Delft University of Technology. Tom Oomen is also with the Department of Mechanical Engineering of Eindhoven  University of Technology, \tt{\{L.Chaouach,D.Boskos\}@tudelft.nl,T.A.E.Oomen@tue.nl}.}

\maketitle

\begin{abstract}
Distributionally robust optimization (DRO) incorporates robustness against uncertainty in the specification of probabilistic models. This paper focuses on mitigating the curse of dimensionality in data-driven DRO problems with optimal transport ambiguity sets. By exploiting independence across lower-dimensional components of the uncertainty, we construct structured ambiguity sets that exhibit a faster shrinkage as the number of collected samples increases. This narrows down the plausible models of the data-generating distribution and mitigates the conservativeness that the decisions of DRO problems over such ambiguity sets may face. We establish statistical guarantees for these structured ambiguity sets and provide dual reformulations of their associated DRO problems for a wide range of objective functions. The benefits of the approach are demonstrated in a numerical example.
\end{abstract}

\section{Introduction}

Uncertainty in decision-making is abundant across engineering and science.  Events with unpredictable outcomes add an additional layer of complexity to the decision-making process in view of the need to strike a balance between performance and risk aversion. To this end, stochastic approaches quantify the uncertainty using a probabilistic model that characterizes the range and frequency of possible outcomes~\cite{KM:15}. In stochastic optimization, the uncertain parameters are typically assumed to follow a known distribution~\cite{AS-DD-AR:14}. This, in turn, guarantees that the solution of the optimization problem enjoys some desired statistical properties. However, in practical scenarios, the true probability distribution is often uncertain and it is hard to infer it from data with sufficient accuracy. Being uncertain about the uncertainty itself can generate unreliable decisions, which may in turn lead to undesirable risks and failures of complex engineered systems. This makes addressing distributional uncertainty a problem of high importance.  

Distributionally robust optimization (DRO) makes decisions in the face of uncertainty without resorting to a single probability distribution. Instead, it robustifies stochastic optimization problems by considering an ambiguity set of plausible models for the unknown distribution of the uncertainty~\cite{WW-DK-MS:14}. This way, DRO hedges against model misspecification due to insufficient or corrupted data, which is the typical situation in real-life systems across engineering, finance, machine learning, medicine, and social sciences. There is, therefore, an increasing interest in exploiting DRO for stochastic decision problems, which are widespread in operations research~\cite{DB-DBB-CC:11}, statistical learning~\cite{DK-PME-VAN-SAS:19,RC-ICP:18}, and control~\cite{BPGVP-DK-PJG-MM:15,BL-YT-AGW-GRD:22}. Toward applications of DRO in control,  \cite{IT-CDC-CNH:21} develops a distributionally robust LQR framework. Data-driven formulations of Wasserstein distributionally robust stochastic control are found in \cite{IY:21, LA-MF-JL-FD:23} and  \cite{JC-JL-FD:21}, while \cite{SSA-VAN-DK-PME:18} provides a Kalman filtering design that accounts for distributional uncertainty. The problem of propagating optimal transport ambiguity sets is considered in \cite{DB-JC-SM:21-tac,DB-JC-SM:23,LA-NL-HC-FD:23}, which take into account multiple data assimilation nonidealities. 
Further applications of DRO include economic dispatch in power systems \cite{BKP-ARH-SB-DSC-AC:20}, congestion avoidance in traffic control~\cite{DL-DF-SM:19-ecc}, and motion planning in dynamic environments \cite{AH-IY:21}.

There are multiple choices of ambiguity sets. In data-driven cases, these choices affect both the statistical properties and the tractability of their associated DRO problems. Typical ambiguity sets are constructed using statistical divergences~\cite{GCC-LEG:06,RJ-YG:16}, moment constraints~\cite{IP:07,ED-YY:10}, total variation metrics~\cite{IT-CDC-TC:15}, and optimal transport discrepancies~\cite{GP-DW:07}, such as the Wasserstein distance~\cite{CV:08}. Among the favorable properties of  Wasserstein ambiguity sets are tractable reformulations of their associated DRO problems~\cite{PME-DK:17,RG-AJK:23,JB-KM:19} and rigorous statistical guarantees \cite{NF-AG:15} of containing the data-generating distribution. 
In particular, for a given confidence level, the size of these ambiguity sets decreases with respect to the number of collected samples \cite{NF-AG:15}. Nevertheless, this decay rate suffers from the curse of dimensionality as it becomes excessively slow with the number of samples for high-dimensional data~\cite{SD-MS-RS:13,NF-AG:15,JW-FB:19}. To ameliorate this drawback, a recent line of work informs the ambiguity set by the specific optimization problem, rendering the ambiguity-size decay rate independent of the dimension of the uncertainty~\cite{JB-YK-KM:16,JB-KM-NS:21,SS-DK-PME:19,RG:22,NS-JB-SG-MS:20}. 
There is also DRO literature, which considers optimal transport ambiguity sets that take into account structural properties of the unknown distribution, like heterogeneity or information about its marginals. To this end, \cite{JB-YK-KM-FZ:19} builds Wasserstein ambiguity balls using a Mahalanobis distance that allocates a higher transport cost to directions with a larger impact on the expected loss, while \cite{JB-KM-FZ:22} considers a state-dependent variant of this distance. A distributionally robust decision framework for ambiguity sets of multivariate distributions with known marginals is provided in \cite{RG-AJK:17}, which encodes dependency variations through the Wasserstein distance, while \cite{DB-MK-TL-AP-SE:22} establishes optimal transport duality for ambiguity sets that are defined through Fr\'echet classes and allow variations of their marginals.  

Although important steps have been taken to develop adequate DRO approaches to address complex data-driven problems, the curse of dimensionality with respect to the dimension of the uncertainty still persists in important classes of problems. These include model predictive control~\cite{PC-PP:22, FW-MEV-BH:22, LA-MF-JL-FD:23}, controller synthesis for stochastic reach-avoid specifications~\cite{IG-DB-LL-MM:23}, and distributionally robust dynamic programming~\cite{IY:21}, which involve solving multiple optimization problems under the same uncertainty.
The aim of this paper is to address the curse of dimensionality that characterizes Wasserstein ambiguity sets when they are accompanied by the requirement to contain the unknown distribution with a prescribed probability. To this end, we build new classes of optimal transport ambiguity sets, which shrink at favorable rates with the number of samples while containing the true distribution with a fixed confidence.
Obtaining these probabilistic guarantees necessitates further assumptions regarding the class to which the distribution belongs. In this paper, we assume independence between lower dimensional components of the random variable and build ambiguity sets with distributions that share similar structural properties. Besides the improved statistical guarantees that accompany these ambiguity sets, which we call \textit{structured ambiguity sets}, we also provide dual reformulations of their corresponding DRO problems.

Our first contribution is the introduction of two classes of structured ambiguity sets, which we call Wasserstein hyperrectangles and optimal-transport hyperrectangles, respectively. The former are designed to contain only product distributions while the latter contain distributions that simultaneously respect multiple optimal transport constraints. Our second contribution is to show that both ambiguity sets shrink faster than traditional Wasserstein balls in data-driven scenarios while containing the true distribution with the same confidence level. This is established under independence of lower-dimensional components of the random variable and breaks the curse of dimensionality when these components are of sufficiently small dimension. Our third contribution is the derivation of dual reformulations of DRO problems associated with these ambiguity sets. Due to the convexity of multi-transport hyperrectangles, which is in principle not shared by Wasserstein hyperrectangles, their DRO problems admit dual reformulations for a much broader class of objective functions. Preliminary results introducing the concept of Wasserstein hyperrectangles have appeared in \cite{LMC-DB-TO:22}. The contributions of the present paper extend far beyond \cite{LMC-DB-TO:22}, including 1) the proofs of the results in \cite{LMC-DB-TO:22} 2) the new notion of multi-transport hyperrectangles and their more elaborate duality theory, which applies to a substantially broader class of objective functions, and, 3) a more complete treatment of the probabilistic guarantees that are associated with both classes of ambiguity sets.   

This paper is organized as follows. In Section~\ref{sec:prelims}, we introduce mathematical preliminaries and notation. We formulate the problem in  Section~\ref{sec:problem:formulation} and introduce two classes of structured ambiguity sets in Section~\ref{sec:hyperrectangles}. In Section~\ref{sec:statistical:guarantees}, we provide probabilistic guarantees for these ambiguity sets and we present dual reformulations for their associated DRO problems in Section~\ref{sec:duality}. In Section~\ref{sec:example}, we illustrate the results of the paper in a simulation example.   

\section{Preliminaries and notation}
\label{sec:prelims}

Throughout this paper, we use the following notation. We denote by $\|\cdot\|_p$ the $p$th norm in $\mathbb{R}^d$ with $p\in[1,\infty]$. We denote by $\mathbb R_{\ge0}$ and $\mathbb R_{>0}$ the positive and strictly positive real numbers, respectively, and define $\bar{\mathbb R}:=\mathbb R\cup\{-\infty,+\infty\}$. For $N\in\mathbb N\backslash\{0\}$, we denote $[N]:=\{1,\ldots,N\}$. The diameter of $S\subset\mathbb{R}^d$ is $\mathrm{diam}(S):=\mathrm{sup}\{\left\lVert x-y\right\rVert_\infty:x,y\in S\}$. We denote by $C(\Xi)$ the class of continuous real-valued functions on a topological space $\Xi$, and by $C_{{\rm const},2}(\Xi\times\Xi)$  the functions  $\gamma\in C(\Xi\times\Xi)$ with  $\gamma(\zeta,\xi)=\gamma(\zeta,\xi')$ for all $\zeta,\xi,\xi'\in \Xi$. Given the set $\Xi=\Xi_1\times\cdots\times\Xi_n$ and $k\in[n]$, we define the projection ${\rm pr}_k:\Xi\to \Xi_k$ as ${\rm pr}_k(\xi):=\xi_k$, for all $\xi=(\xi_1,\ldots,\xi_n)\in\Xi$, and define analogously ${\rm pr}_{k,l}$ when projecting to two components indexed by $k,l\in[n]$. Given a normed linear space $X$ and its topological dual $X^*$, the conjugate of a  function $h:X \to\mathbb R\cup\{+\infty\}$ is defined by $h^*(x^*):=\sup_{x\in X}\{\langle x^*,x\rangle-h(x)\}$. For a vector space $X$ and a convex cone $K\subset X$, we denote by $\succeq_K$ the order with respect to $K$, given by $x\succeq_K y$ iff $x-y\in K$ and will omit the dependence on $K$ when it is clear from the context. For example, the order  $\succeq$ in $\Rat{d}$ with respect to the positive cone  $\Rat{d}_{\ge 0}:=\{(x_1,\ldots,x_d)\in\Rat{d}:x_k\ge 0\;\textup{for all}\;k\in[d]\}$ implies that $x\succeq y$ iff $x_k\ge y_k$ for all $k\in[d]$. 

\textit{Probability theory:} Let $\Xi$ be a Polish space, namely, a  complete and separable metric space. We denote by $\rho$ the metric on $\Xi$, by $\mathcal{B}(\Xi)$ its Borel $\sigma$-algebra, and by $\mathcal{P}(\Xi)$ the space of probability measures on $(\Xi,\mathcal{B}(\Xi))$. The Dirac distribution centered at $\xi\in\Xi$ is denoted by $\delta_\xi$. The indicator function $\mathds 1_\Theta$ of $\Theta\subset\Xi$ is $\mathds 1_{\Theta}(\xi):=1$ if $\xi\in\Theta$ and $0$ otherwise. 
Given the measurable spaces $(\Omega,\mathcal F)$ and $(\Omega',\mathcal F')$, a measurable map $\Psi: (\Omega,\mathcal F) \to (\Omega',\mathcal F')$ assigns to each (signed) measure $\mu$ in $(\Omega,\mathcal F)$ the pushforward measure $\Psi_\#\mu$ in $(\Omega',\mathcal F')$ defined by $\Psi_\#\mu(B):=\mu(\Psi^{-1}(B))$ for all $B\in\mathcal F'$. We denote by $P\otimes Q$ the product measure of $P$ and $Q$. For any $P\in\mathcal{P}(\Xi)$, its support is the closed set $\mathrm{supp}(P):=\{x\in\Xi:P(U)>0\;\text{for each neighborhood}\;U\;\text{of}\;x\}$. Given a function $X:\Omega\to\Xi$ with the $\sigma$-algebra $\mathcal B(\Xi)$ we denote by $\sigma(X)$ the $\sigma$-algebra generated by $X$ on $\Omega$. The universal $\sigma$ algebra on $\Xi$ is defined as $\mathcal U(\Xi):=\cap_{P\in\P(\Xi)}\mathcal B_P(\Xi)$ (cf. \cite[Definition 7.18]{DB-SES:96}), where $\mathcal B_P(\Xi)$ refers to the completion of the $\sigma$-algebra $\mathcal B(\Xi)$ with respect to the measure $P$ (cf. \cite[Remark 1.70]{AK:13}) and  satisfies $\mathcal B(\Xi)\subset\mathcal B_P(\Xi)$. We denote by $\mathfrak m_{\mathcal U}(\Xi;\Rat{}\cup\{+\infty\})$ the space of measurable functions from $(\Xi,\mathcal U(\Xi))$ to $\Rat{}\cup\{+\infty\}$ with its Borel $\sigma$-algebra. 
For any $p\geq1$, we denote by $\mathcal{P}_p(\Xi)$ the set of probability measures in $\mathcal{P}(\Xi)$ with finite $p$th moment. Given $P,Q\in\mathcal{P}_p(\Xi)$, their $p$th Wasserstein distance is 
\begin{align*}
    W_{p}(Q,P):=\underset{\pi\in \mathcal C(Q,P)}{\mathrm{inf}} \left\{\int_{\Xi\times\Xi} \rho(\zeta,\xi)^pd\pi(\zeta,\xi)\right\}^{\frac{1}{p}}
\end{align*}
(cf.~\cite{CV:08}). Each $\pi\in\mathcal C(Q,P)$ is a transport plan, i.e., a distribution on $\Xi\times\Xi$ with marginals $P=\rm{pr}_{2\#}\pi$ and $Q=\rm{pr}_{1\#}\pi$, respectively. The Wasserstein distance between $P$ and $Q$ is defined through the optimal cost to transfer the mass of one distribution to the other when the cost to transfer a unit of mass between two locations $\zeta$ and $\xi$ in $\Xi$ is $\rho(\zeta,\xi)^p$. By Katorovich duality (cf. \cite[Theorem 1.3]{CV:08}), the optimal transportation cost $W_{p}^p(Q,P)$ is equal to the value of its dual optimization problem 
\begin{align*}
   K_p(Q,P)=\sup_{\substack{(\psi,\phi)\in {L}^1(Q)\times  {L}^1(P)\\ \phi(\xi)-\psi(\zeta) \le \rho(\zeta,\xi)^p}}\Big\{ \int_\Xi \phi(\xi)dP(\xi)-\int_\Xi \psi(\zeta) dQ(\zeta)\Big\}.
\end{align*}

\section{Problem formulation}
\label{sec:problem:formulation}

In this section, we introduce data-driven stochastic optimization problems and their distributionally robust formulations that hedge against model uncertainty. 
Consider the stochastic optimization problem 
\begin{equation}
\underset{x\in\mathcal{X}}{
\mathrm{inf}}\mathbb{E}_{P_\xi}\big[f(x,\xi) \big], \label{problem 1}
 \end{equation}
 where $f$ is the objective function, $x\in\mathcal{X}$ is the decision variable, and $\xi$ is a random variable, which takes values in a Polish space $\Xi$ and has distribution $P_\xi$. 
 
A typical situation that fits into \eqref{problem 1} is when the distribution $P_\xi$ is unknown and there is only access to a finite number of i.i.d. samples $\xi^1,\ldots,\xi^N$ of $\xi$. The usual approach to approximate the solution of  \eqref{problem 1} in this case is to replace $P_\xi$ by the empirical distribution $P_\xi^N:=\frac{1}{N}\sum_{i=1}^N\delta_{\xi^i}$. This is known as the Sample Average Approximation (SAA) of \eqref{problem 1} and it converges to the solution of the original problem in the asymptotic limit~\cite{AS-DD-AR:14}.
  
 \subsection{Distributionally robust optimization}
 
When the available data are limited, the empirical distribution $P_\xi^N$ may exhibit significant deviations from the true distribution $P_\xi$, which can in turn have a considerable impact on the discrepancy between the SAA and the original optimization problem. To address this issue, uncertainty in the distribution is incorporated into \eqref{problem 1} under the robust formulation
 \begin{equation}
     \inf_{x\in\mathcal X} \underset{P\in\mathcal{P}^N}{\mathrm{sup}}
     \mathbb{E}_{P}\big[f(x,\xi) \big]. \label{DR problem}
 \end{equation}
 In this distributionally robust optimization (DRO) problem, $\mathcal{P}^N$ is an ambiguity set of distributions that is inferred from the samples and contains plausible models of the true distribution. 
 
A well-established approach to construct data-driven ambiguity sets is to group all distributions that are $\varepsilon$-close to the empirical distribution $P_\xi^N$ in the $p$th Wasserstein metric for some $p\ge 1$ and $\eps>0$. In this case, $\mathcal{P}^N$ in \eqref{DR problem} is the ball 
\begin{align*}
\mathcal B_p(P_\xi^N,\varepsilon):=\{P\in\mathcal 
P_p(\Xi):W_p(P_\xi^N,P)\le\varepsilon\}
\end{align*}
with center $P_\xi^N$ and radius $\varepsilon$. Among the benefits of this choice are that Wasserstein distances yield higher penalties to distributional variations that are farther apart in the domain, which typically induce larger discrepancies on the optimization problems, and that Wasserstein balls lead to tractable DRO problems \cite{PME-DK:17}. In addition, for any number of samples, one can tune the radius of a Wasserstein ball so that it contains the true distribution with prescribed confidence. 
As a result, the value of \eqref{DR problem} provides an upper bound for the expected cost \eqref{problem 1} with prescribed confidence.   

\subsection{Structured ambiguity sets}  

The size of the ambiguity set $\mathcal{P}^N$ has a direct effect on the solution of \eqref{DR problem} since ambiguity balls of larger sizes may lead to conservative upper bounds for \eqref{problem 1}. This can happen because an ambiguity ball that is sufficiently large to contain the true distribution with a prescribed probability may also contain several irrelevant distributions. To address this issue, we consider some prior knowledge about the uncertainty, which can facilitate the construction of ambiguity sets whose elements are more appropriate models of the unknown distribution. We make the following assumption regarding the class of the random variable.
    
\begin{assumption}
\longthmtitle{Independent random variable components} (i) The random  variable $\xi$ takes values in the Polish space $\Xi=\Xi_1\times\cdots\times\Xi_n$, with $\Xi$ and $\Xi_k$, $k\in[n]$ equipped with the metrics $\rho$ and $\rho_k$, $k\in[n]$, respectively. (ii) The components $\xi_1,\ldots,\xi_n$ of $\xi$ are independent.
 \label{assumption:independence}
\end{assumption}
This assumption is reasonable in several problems such as in networked systems, where random inputs at different network locations do not essentially affect each other, or the deployment of multi-robot systems where individual agents are subject to independent disturbances. 

\noindent \textbf{Problem formulation.} Under Assumption~\ref{assumption:independence}, we seek to introduce  \textit{structure} in data-driven ambiguity sets so that they contain the true distribution with high probability while excluding implausible distributions and enabling the formulation of tractable DRO problems.

To this end, note that due to Assumption \ref{assumption:independence}, the distribution of $\xi$ is the product measure  
\begin{equation}
    P_{\xi}=P_{\xi_1}\otimes\cdots\otimes P_{\xi_n}, \label{distribution class}
\end{equation}
with $P_{\xi_k}$, $k\in[n]$ denoting the distributions of its components. Thus, instead of looking for plausible models of $P_\xi$ in an ambiguity ball, we can consider ambiguity sets whose distributions are only product measures, or at least sufficiently close to product measures. Such a  set should contain a restricted class of distributions, and therefore, yield less conservative solutions for \eqref{DR problem} under the same confidence. 

\subsection{Ambiguity radius}
 
By tuning the radius of the ambiguity ball, it is possible to guarantee that it contains the true distribution with prescribed probability. These guarantees hinge on concentration of measure results, which leverage prior assumptions about the class where the unknown distribution belongs, to bound the Wasserstein distance between the true and the empirical distribution. Such assumptions are the size of the distribution's support (e.g., \cite[Proposition 10]{NF-AG:15}, \cite{JW-FB:19}), its tail decay rate (e.g., \cite[Theorem 2, cases (1) and (2)]{NF-AG:15}), or bounds on its moments (e.g., \cite[Theorem 2, case (3)]{NF-AG:15}, \cite{JD-FM:19}). Based on these results, for any confidence $1-\beta$ and number $N$ of i.i.d. samples, we can select the ambiguity radius $\varepsilon(N,\beta)$ so that 
\begin{align} \label{confidence:bound:ball}
\mathbb P(P_\xi\in\mathcal B_p(P_\xi^N,\varepsilon))\ge 1-\beta.
\end{align}
The radius can typically be determined by a bound of the form  
\begin{align}
\varepsilon(N,\beta)\le K\frac{1}{N^{1/\max\{d,2p\}}},
\end{align}
where $d$ is the dimension of the random vector $\xi$. Therefore, for high-dimensional random variables, the decrease of the radius with respect to the number of samples becomes excessively slow. As a result, the exploitation of more data does not guarantee any significant improvement of the closeness between the true distribution and its empirical approximation, and hence, also of the size of the ambiguity ball. In this regard, we seek to \textit{exploit the independence Assumption~\ref{assumption:independence} for the components of $\xi$ and determine an ambiguity set structure that does not suffer from the curse of dimensionality with respect to $d$.}     
\section{Ambiguity hyperrectangles}
\label{sec:hyperrectangles}

In this section, we introduce two classes of structured ambiguity sets and provide some of their key statistical properties for data-driven problems. The starting point to construct these ambiguity sets are the lower-dimensional components of the random variable $\xi=(\xi_1,\ldots,\xi_n)$. Using $N$ i.i.d. samples $\xi^1,\ldots,\xi^N$, we first  build a lower-dimensional ambiguity ball $\mathcal B_p(P_{\xi_k}^N,\varepsilon_k)$ for each component of $\xi$, where $P_{\xi_k}^N:=\frac{1}{N}\sum_{i=1}^N\delta_{\xi_k^i}$ denotes its corresponding empirical distribution. From these balls, we construct the \textit{Wasserstein hyperrectangle}
\begin{subequations} \label{hyperrectangle:dfn:and:center}
\begin{align}
\mathcal H_p(\bm P_\xi^N,\bm\varepsilon):= \; & \{  P_{\xi_1}'\otimes\cdots\otimes P_{\xi_n}': P_{\xi_k}'\in  \mathcal B_p  (P_{\xi_k}^N,\varepsilon_k)\;\textup{for all}\;k\in[n]\} \label{hyperrectangle} \\
\bm P_\xi^N:= \; & P_{\xi_1}^N\otimes\cdots\otimes P_{\xi_n}^N,
\quad \bm\varepsilon=(\varepsilon_1,\ldots,\varepsilon_n), \label{rectangle:center:and:radius}
\end{align}
\end{subequations}
by taking the product measures across the individual distributions from the balls. We refer to the nominal model $\bm P_\xi^N$ around which the ambiguity set is built as the \textit{product empirical distribution}. 

Next, we establish probabilistic guarantees, which ensure that the Wasserstein hyperrectangles contain the distribution of $\xi$ with prescribed confidence. Later, we exploit these guarantees to alleviate the curse of dimensionality regarding the shrinkage of Wasserstein balls. The following result establishes the guarantees that a Wasserstein hyperrectangle inherits from its lower-dimensional constituent ambiguity balls when the components of $\xi$ are independent.

\begin{thm}\label{thm:hyperrectangle:confidence}
\longthmtitle{Probabilistic guarantees for Wasserstein hyperrectangles}
Assume that the random variable $\xi$ satisfies Assumption \ref{assumption:independence} and that $P_\xi\in\mathcal P_p(\Xi)$. 
Given i.i.d. samples $\xi^1,\ldots,\xi^N$ of $\xi$, let $P_{\xi_1}^N,\ldots,P_{\xi_n}^N$ be the empirical distributions of the individual components. Assume also that each Wasserstein ball $\mathcal B_p(P_{\xi_k}^N,\varepsilon_k)$ contains $P_{\xi_k}$ with confidence  $1-\beta_k$. Then the Wasserstein hyperrectangle $\mathcal{H}_p(\bm P_{\xi}^N,\bm\varepsilon)$ contains $P_\xi$ with confidence $\prod_{k=1}^n1-\beta_k$. 
\end{thm}

To prove Theorem \ref{thm:hyperrectangle:confidence} we use the following lemma, whose proof is given in Appendix~\ref{appendix:to:sec:hyperrectangles}.

\begin{lemma}
\longthmtitle{Independent Wasserstein distances across empirical distributions}
Assume that the random variable $\xi$ satisfies Assumption \ref{assumption:independence} and that $P_\xi\in\mathcal P_p(\Xi)$. Given i.i.d. samples $\xi^1,\ldots,\xi^N$ of $\xi$, let $P_{\xi_1}^N,\ldots,P_{\xi_n}^N$ be the empirical distributions of its  components. Then for any $\varepsilon_1,\ldots,\varepsilon_n\ge 0$ the events $\{W_p(P_{\xi_k}^N,P_{\xi_k})\le\varepsilon_k\}$, $k\in[n]$ are independent. \label{lemma independence}
\end{lemma}

\begin{proof}[Proof of Theorem \ref{thm:hyperrectangle:confidence}]
By Assumption~\ref{assumption:independence}, $P_\xi$ is expressed as the product distribution in \eqref{distribution class}. Thus, we get from the definition of the Wasserstein hyperrectangle in  \eqref{hyperrectangle:dfn:and:center} that  
\begin{align}
      \mathbb{P}\big(P_\xi\in \mathcal{H}_p(\bm P_\xi^N,\bm \varepsilon)\big) 
      =\mathbb{P}\big(P_{\xi_k}\in \mathcal B_p(P_{\xi_k}^N,\varepsilon_k)\;\textup{for all}\; k\in[n]\big).
    \label{joint prob proof thm V.1}
\end{align}
Also, by the definition of a Wasserstein ball,  
\begin{align}
    P_{\xi_k}\in \mathcal B_p(P_{\xi_k}^N,\varepsilon_k)\Longleftrightarrow W_p(P_{\xi_k}^N,P_{\xi_k})\leq\varepsilon_k. \label{proof thm V.1 eq12}
\end{align}
Since the components of $\xi$ are independent, we get from the independence result of  Lemma \ref{lemma independence}, \eqref{joint prob proof thm V.1}, and \eqref{proof thm V.1 eq12} that 
\begin{align}
    \mathbb{P}\big(P_\xi\in \mathcal{H}_p(\bm P_\xi^N,\bm\varepsilon)\big)  
    =\prod_{k=1}^n\mathbb{P}\big(P_{\xi_k}\in \mathcal B_p(P_{{\xi}_k}^N,\varepsilon_k)\big). \label{prod prob proof thm V.1}
\end{align}
Recalling that the $k$th Wasserstein ball contains $P_{\xi_k}$ with confidence $1-\beta_k$ for each  $k\in[n]$, we get from \eqref{prod prob proof thm V.1} that the hyperrectangle contains $P_\xi$ with confidence $\prod_{k=1}^n1-\beta_k$, which concludes the proof. 
\end{proof}

\begin{rem}
\longthmtitle{Boldface notation}
We use boldface notation throughout the paper to signify elements, which in contrast to the typical DRO literature, admit a vectorized or product representation. These include the vector Wasserstein radii $\bm\varepsilon$, the product empirical distribution $\bm P_\xi^N$ ---to distinguish it from the standard empirical distribution $P_\xi^N$---, and vectors of dual variables $\bm\lambda$ that are introduced later in dual DRO reformulations. 
\end{rem}

Note that we can directly generalize the notion of a Wasserstein hyperrectangle to the case where the nominal distribution is a general product distribution $Q=Q_1\otimes\cdots\otimes Q_n$ on the Polish space  $\Xi=\Xi_1\times\cdots\times\Xi_n$ with $Q_k\in\P(\Xi_k)$ for each $k\in[n]$, instead of the product empirical distribution $\bm P_\xi^N$. Again, the Wasserstein hyperrectangle $\mathcal H_p(Q,\bm \eps)$ comprises of all product distributions whose $k$th lower-dimensional marginal has Wasserstein distance at most $\eps_k$ from the corresponding marginal of $Q$. Since Wasserstein hyperrectangles contain only product distributions, they are non-convex. This restricts the class of cost functions for which the DRO problem \eqref{DR problem} with $\P^N\equiv\mathcal H_p(Q,\bm\varepsilon)$ admits tractable reformulations. To overcome this obstacle, we build a convex ambiguity set, which shrinks at the same favorable rate as the Wasserstein hyperrectangle with respect to the number of samples. The distributions of this ambiguity set are defined through couplings with a nominal distribution, which need to respect a set of transport cost constraints. 

In particular, consider a general reference distribution $Q\in\mathcal P(\Xi)$ and let  
\begin{align*}
\Pi(Q):=\{\pi\in\P(\Xi\times\Xi): {\rm pr_{1\#}}\pi=Q\}.
\end{align*}
Consider also the lower semicontinuous cost functions $c_k:\Xi\times\Xi\to\mathbb R_{\ge 0}$, $k\in[n]$ with $c_k(\zeta,\zeta)=0$ for all $\zeta\in\Xi$, the transport budget vector $\bm\epsilon=(\epsilon_1,\ldots,\epsilon_n)$ with positive entries, and let
\begin{align}
 \Pi(Q,\bm\epsilon)\equiv\Pi(Q,\bm\epsilon;c_1,\ldots,c_n):=\bigg\{\pi\in\Pi(Q):\int_{\Xi\times\Xi} c_k(\zeta,\xi)d\pi(\zeta,\xi)\le\epsilon_k\;\textup{for all}\;k\in[n]\bigg\}.
 \label{transport:plan:set}
\end{align}
Due to the fact that each $c_k(\zeta,\zeta)\equiv0$, $\Pi(Q,\bm\epsilon)$ is always nonempty. We define the \textit{multi-transport hyperrectangle}
\begin{align}
\mathcal T(Q,\bm\epsilon):={\rm pr}_{2\#}\Pi(Q,\bm\epsilon), \label{transport hyperrectangle}
\end{align}
which is convex and depends on the chosen cost functions $c_1,\ldots,c_n$. When Assumption \ref{assumption:independence}(i) is satisfied and  the costs are $c_k(\zeta,\xi):=\rho_k(\zeta_k,\xi_k)^p$ for some $p\ge 1$, we denote 
\begin{align}
\mathcal T_p(Q,\bm\eps):=\mathcal T(Q,\bm\eps^p), \label{transport:hyperrectangle:Tp}
\end{align}
where $\bm\eps^p:=(\eps_1^p,\ldots,\eps_n^p)$.
The next result delineates the relation between Wasserstein hyperrectangles and multi-transport hyperrectangles of the form \eqref{transport:hyperrectangle:Tp} that are built around product distributions. 

\begin{prop}
\label{prop:containment}
\longthmtitle{Wasserstein hyperrectangle containment}  
Consider a Polish space $\Xi$ as in Assumption~\ref{assumption:independence}(i) and a product distribution $Q=Q_1\otimes\cdots\otimes Q_n\in\P_p(\Xi)$ with $Q_k\in\mathcal P_p(\Xi_k)$ for each $k\in[n]$. Then $\mathcal H_p(Q,\bm\varepsilon)\subset\mathcal T_p(Q,\bm\eps)$. In addition, for any product distribution $P\in\mathcal T_p(Q,\bm\eps)$,  also $P\in\mathcal H_p(Q,\bm\eps)$.
\end{prop}

\begin{proof}
Let $P\in \mathcal{H}_p(Q,\boldsymbol\varepsilon)$. Then $P=P_1\otimes\cdots\otimes P_n$ and $P_k\in\mathcal B_p(Q_k,\varepsilon_k)$ for all $k\in[n]$, which implies that $W_p(Q_k,P_k)\le\varepsilon_k$.
Thus, for each $k\in[n]$, there exists an optimal transport plan $\pi_k$ for the Wasserstein distance between $Q_k$ and $P_k$ (cf. \cite[Theorem 4.1]{CV:08}) with   
\begin{align} \label{individual:OT:constraints}
\int_{\Xi_k\times\Xi_k} \rho_k(\zeta_k,\xi_k)^p d\pi_k(\zeta_k,\xi_k)\le\varepsilon_k^p.
\end{align}
Next, define 
\begin{align}
\label{special:transport:plan}
\pi:=\bigotimes_{k=1}^n\pi_k\quad{\rm and}\quad\widetilde{\pi}:=T_{\#}\pi,
\end{align}
where $T:\prod_{k=1}^n \Xi_k\times\Xi_k \to \prod_{k=1}^n \Xi_k \times\prod_{k=1}^n \Xi_k $ is the linear map $T(\zeta_1,\xi_1,\ldots,\zeta_n,\xi_n):=(\zeta_1,\ldots,\zeta_n,\xi_1,\ldots,\xi_n)$. Then $\widetilde{\pi}$ is a transport plan between $Q$ and $P$ since
\begin{align*}
\widetilde{\pi}(A_1\times\cdots\times A_n\times\Xi)
&=\widetilde{\pi}(A_1\times\cdots\times A_n\times\Xi_1\times\cdots\times\Xi_n)
\overset{(a)}{=}\pi(A_1\times\Xi_1\times\cdots\times A_n\times\Xi_n)\\
&\overset{(b)}{=}\prod_{k=1}^n\pi_k(A_k\times\Xi_k)\overset{(c)}{=}\prod_{k=1}^n Q_k(A_k)
=Q_1\otimes\cdots\otimes Q_n(A_1\times\cdots\times A_n)\\
&=Q(A_1\times\cdots\times A_n)
\end{align*}
for any $A_k\in\mathcal B(\Xi_k)$, $k\in[n]$. Here we used \eqref{special:transport:plan} in (a) and (b), and the fact that each $\pi_k$ is a transport plan in (c). Analogously, $P$ is also a marginal of $\pi$. In addition, we get from \eqref{individual:OT:constraints} and Fubini's theorem (cf.~\cite[Page 233]{PB:08}) that 
\begin{align*}
\int_{\Xi_k\times\Xi_k} \rho_k(\zeta_k,\xi_k)^p d\widetilde\pi(\zeta,\xi)\le \eps_k^p
\end{align*}
for each $k\in[n]$, which by \eqref{transport:plan:set}-\eqref{transport:hyperrectangle:Tp} implies that also $P\in\mathcal T_p(Q,\boldsymbol\varepsilon)$ and concludes the proof of the first claim. 

For the proof of the second claim consider a product distribution $P\in\mathcal T_p(Q,\bm\eps)$. Then there exists a transport plan $\pi$ between $Q$ and $P$ so that \eqref{transport:plan:set} holds with $c_k(\zeta,\xi)\equiv\rho_k(\zeta_k,\xi_k)^p$ and $\epsilon_k\equiv\varepsilon_k^p$. Next, let $\pi_k:={\rm pr}_{k,n+k\#}\pi$ (with $\pi$ viewed as a distribution on $\prod_{k=1}^n \Xi_k \times\prod_{k=1}^n \Xi_k$). It follows that $\pi_k$ has marginals $Q_k$ and $P_k$, respectively, and that it satisfies \eqref{individual:OT:constraints}. As a result, $P_k\in\mathcal B_p(Q_k,\varepsilon_k)$ for all $k\in[n]$ and we conclude that also $P\in\mathcal H_p(Q,\bm\eps)$.
\end{proof}

The next result follows directly from Proposition \ref{prop:containment} and provides conditions under which a multi-transport hyperrectangle contains the true distribution with prescribed confidence.

\begin{corollary}(\textbf{Probabilistic guarantees for multi-transport hyperrectangles})
Assume that the random variable $\xi$ satisfies Assumption~\ref{assumption:independence} and that $P_\xi\in\mathcal P_p(\Xi)$. Given i.i.d. samples $\xi^1,\ldots,\xi^N$ of $\xi$, let $P_{\xi_1}^N,\ldots,P_{\xi_n}^N$ be the empirical distributions of its components. Assume also that each Wasserstein ball $\mathcal B_p(P_{\xi_k}^N,\varepsilon_k)$ contains $P_{\xi_k}$ with confidence  $1-\beta_k$. Then the multi-transport hyperrectangle $\mathcal{T}_p(\bm P_\xi^N,\bm\eps)$ contains $P_\xi$ with confidence $\prod_{k=1}^n1-\beta_k$. \label{cor:confidence}
\end{corollary}

We conclude this section with a result that compares the size of multi-transport hyperrectangles with that of monolithic balls. Specifically, we determine the radius that a Wasserstein ball should have in order to contain a multi-transport hyperrectangle when its reference distribution is also the center of that ball. For this, we also need to relate the metric $\rho$ on the product space $\Xi$ with the metrics $\rho_k$ on the components $\Xi_k$.  
\begin{prop} \longthmtitle{Size of enclosing Wasserstein ball}
\label{prop:hyperrectangles:geomtery}
Let $Q\in \mathcal P_p(\Xi)$ and assume that the metric on $\Xi$ is  
\begin{align} \label{product:metric}
\rho(\zeta,\xi):=\Big(\sum_{k=1}^n\rho_k(\zeta_k,\xi_k)^q\Big)^\frac{1}{q}, \quad\zeta,\xi\in\Xi,    
\end{align}
for some $q\ge 1$. Then the multi-transport hyperrectangle  $\mathcal T_p(Q,\bm\eps)$ satisfies 
\begin{align*}
\mathcal T_p(Q,\bm\eps)\subset \mathcal B_p(Q,\eps), 
\end{align*}
where $\eps=n^{\max\{0,1/q-1/p\}} \big(\sum_{k=1}^n \eps_k^p \big)^\frac{1}{p}$. If in addition $p=q$ and $Q=Q_1\otimes\cdots\otimes Q_n$ is a product distribution, then there exists a product distribution $P\in\mathcal T_p(Q,\bm\eps)$ with $W_p(Q,P)=\eps$.  
\end{prop}

The proof is given in Appendix~\ref{appendix:to:sec:hyperrectangles}. When $Q$ is a product distribution and \eqref{product:metric} holds, we get from the first parts of Propositions~\ref{prop:containment} and~\ref{prop:hyperrectangles:geomtery} the inclusions
\begin{align*}
\mathcal H_p(Q,\bm\eps)\subset\mathcal T_p(Q,\bm\eps)\subset \mathcal B_p(Q,\eps), 
\end{align*}
where the radius $\eps$ of the ball is given in Proposition~\ref{prop:hyperrectangles:geomtery}. From the second part of the same propositions, it follows that when the exponents of the product metric and the Wasserstein distance coincide, there is at least one common point in both ambiguity hyperrectangles that lies on the boundary of their enclosing ball (cf. Figure~\ref{fig:enclosing:ball}).

\begin{figure}[h]
\centering
\includegraphics[width=.55\linewidth]{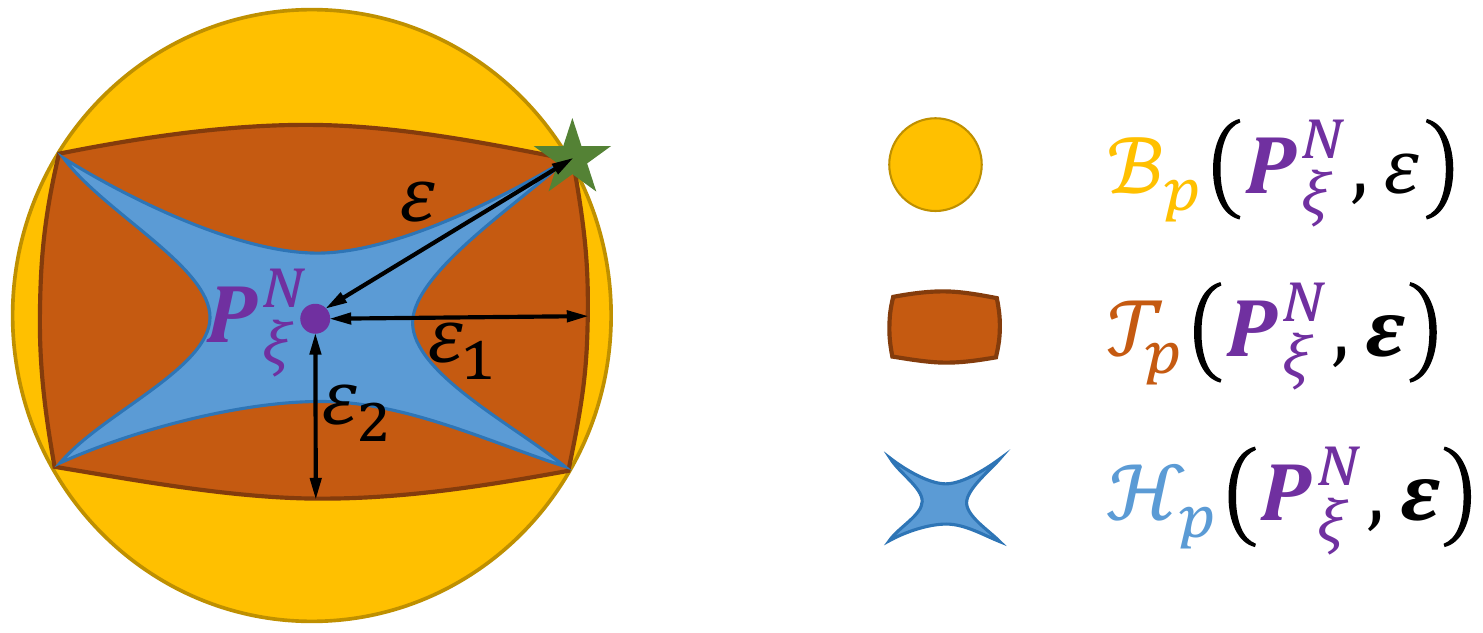} \\ 
\caption{The figure shows the Wasserstein hyperrectangle $\mathcal H_p(\bm P_\xi^N,\bm\varepsilon)$, the multi-transport hyperrectangle $\mathcal T_p(\bm P_\xi^N,\bm\varepsilon)$, and their enclosing ball $\mathcal B_p(\bm P_\xi^N,\varepsilon)$ around the product empirical distribution $\bm P_\xi^N$ for a random variable with two independent components. The star denotes a common distribution of both hyperrectangles that lies on the boundary of the ball.} 
\label{fig:enclosing:ball}
\end{figure}

\section{Ambiguity hyperrectangle size based on the number of samples}
\label{sec:statistical:guarantees}

In this section, we compare Wasserstein and multi-transport hyperrectangles with Wasserstein balls in terms of the size reduction that they exhibit with the number of samples. For this comparison, we assume that both sets are constructed using the same samples and the same confidence level. Since most of the concentration of measure results for this purpose are formulated for distributions supported on Euclidean spaces (cf. \cite{NF-AG:15, DB-JC-SM:23}), we focus on the case where $\Xi$ is a bounded subset of $\Rat{d}$ with the distance induced by the norm $\|\cdot\|_q$. The next result presents bounds for the Wasserstein distance between the true and the empirical distribution, which we exploit to tune the size of $\mathcal H_p(\bm P_\xi^N,\bm \eps)$ and  $\mathcal T_p(\bm P_\xi^N,\bm \eps)$ so that they contain the true distribution with a desired confidence.    

\begin{prop} \label{prop:compact:radius}
\longthmtitle{Ambiguity radius \cite[Proposition 24]{DB-JC-SM:23}}
Assume that the probability distribution $P_\xi$ is supported on $\Xi\subset\mathbb R^d$ with $\rho_\Xi:={\rm diam}(\Xi)<\infty$. 
Assume also that $d\ge 2p+1$ and let $\xi^1,\ldots,\xi^N$  be i.i.d. samples  of $\xi$. Then the ambiguity radius
\begin{align}
\widehat\eps(N,\beta,\rho_\Xi,p,q,d) :=\rho_\Xi\widehat C(\beta,p,q,d)\frac{1}{N^{1/d}}, \label{explicit:ambiguity:radius}
\end{align}
where 
\begin{align*}
\widehat C(\beta,p,q,d) & :=d^{1/q}2^{1/2p}(C(d,p)+(\ln\beta^{-1})^{1/2p}) \\
C(d,p) & :=2^{(d-2)/2p}\Big(\frac{1}{2^{1/2}-1} +\frac{1}{2^{1/2}-2^{1/2-p}}\Big)^{1/p},
\end{align*}
and $1-\beta$ is a desired confidence level, guarantees that
\begin{align*}
\mathbb P(P_\xi\in\mathcal B_p(P_\xi^N,\varepsilon))\ge 1-\beta.
\end{align*}
\end{prop}

Using this ambiguity radius and the guarantees of Theorem~\ref{thm:hyperrectangle:confidence} we determine a ball around the product empirical distribution that contains both the Wasserstein hyperrectangle and the multi-transport hyperrectangle with prescribed probability. The proof of this result is given in Appendix~\ref{appendix:to:sec:statistical:guarantees}. 

\begin{prop} \label{prop:hyperrectangle:containment}
\longthmtitle{Size reduction of ambiguity hyperrectangles}
Assume that the random variable $\xi$ is supported on the compact set $\Xi\equiv \Xi_1\times\cdots\times\Xi_n\subset\Rat{d_1}\times\cdots\times\Rat{d_n}\equiv\mathbb R^d$ with $d_k\ge 2p+1$ for each $k\in[n]$ and satisfies Assumption~\ref{assumption:independence} with the metric induced by $\|\cdot\|_q$ in each space. For any confidence $1-\beta$,  let 
\begin{align*}
\beta_k:=\beta\frac{d_k}{d},\quad   \eps_k:=\widehat\eps(N,\beta_k,\rho_\Xi,p,q,d_k), 
\end{align*} 
with $\widehat\eps$  as in  \eqref{explicit:ambiguity:radius}, and consider the ambiguity sets $\mathcal H_p(\bm P_\xi^N,\bm\eps)$ and $\mathcal T_p(\bm P_\xi^N,\bm\eps)$. Then both sets contain $P_\xi$ with confidence $1-\beta$ and  
\begin{align}
\mathcal H_p(\bm P_\xi^N,\bm\eps)\subset\mathcal T_p(\bm P_\xi^N,\bm  \eps)\subset\mathcal B_p(\bm P_\xi^N,\varepsilon),
\end{align}
where 
\begin{align}
\label{epsilon:decay:rectangles}
\varepsilon=cn^{1/p+\max\{0,1/q-1/p\}}\rho_\Xi\widehat C(\beta,p,q,d)\frac{1}{N^{1/d_{\max}}},
\end{align}
$c:=(\sqrt{2q+1}+1)/\big(2e^{(\sqrt{2q+1}+1)^2/8}\big)$, $d_{\max}:=\max_{k\in[n]}d_k$, and $\widehat C$ is defined in Proposition~\ref{prop:compact:radius}. 
\end{prop} 

\begin{figure}[h]
	\centering
	\includegraphics[width=.6\linewidth]{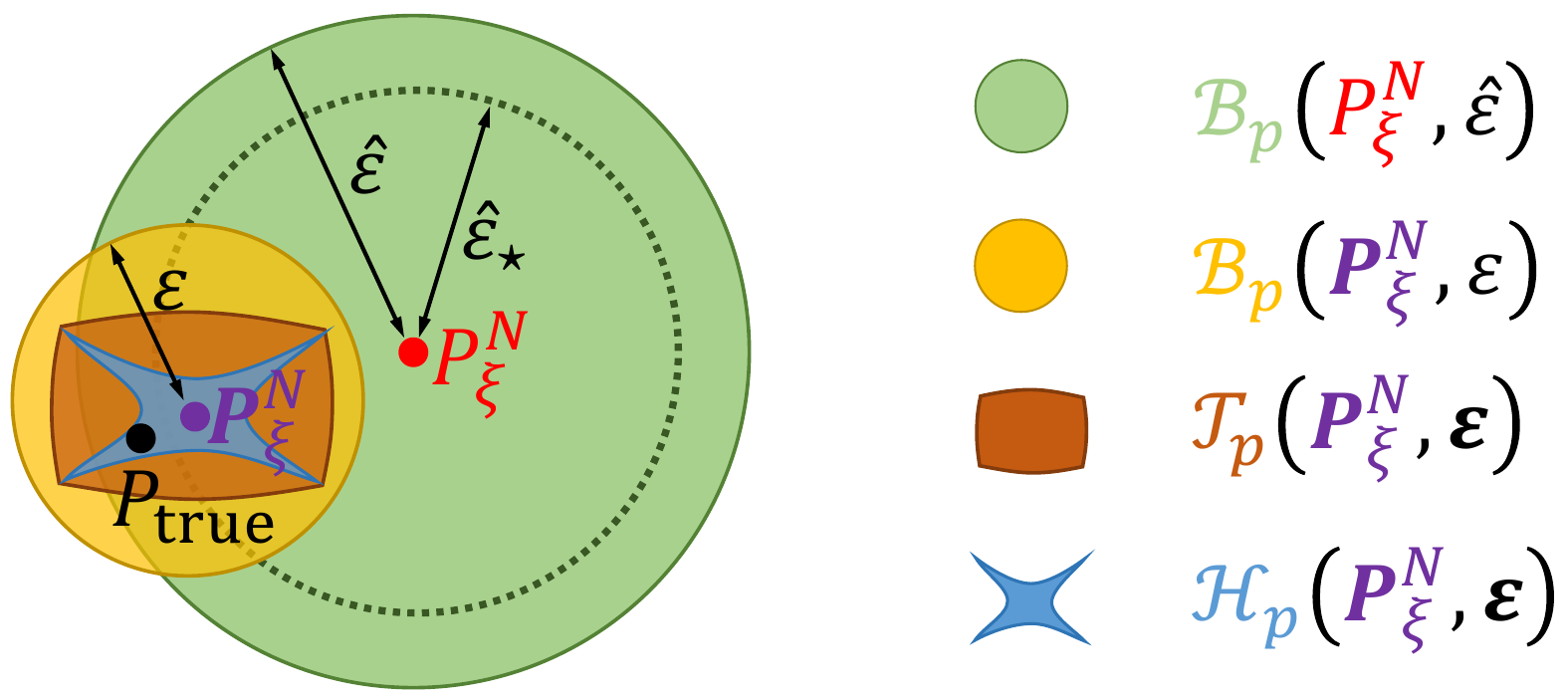} \\ 
	\caption{The figure shows the  hyperrectangles $\mathcal H_p(\bm P_\xi^N,\bm\varepsilon)$ (in blue),  $\mathcal T_p(\bm P_\xi^N,\bm\varepsilon)$ (in dark red), and their enclosing ball $\mathcal B_p(\bm P_\xi^N,\varepsilon)$ with $\eps$ given in \eqref{epsilon:decay:rectangles} (in yellow), which are all centered at the product empirical distribution $\bm P_\xi^N$, as well as the monolithic Wasserstein ball $\mathcal B_p(P_\xi^N,\eps)$ (in green) around the empirical distribution $P_\xi^N$. All  sets contain the true distribution, which is always outside the dashed ball around  $P_\xi^N$ with radius equal to the lower bound $\widehat\eps_\star$ in \eqref{eq:lower:bound1}.} \label{fig:rectangles:vs:balls}
\end{figure}

Under the assumptions of Proposition~\ref{prop:hyperrectangle:containment}, we can compare the size of both hyperrectangles and a monolithic ball that contains $P_\xi$ with the same confidence. If we use the bounds of Proposition \ref{prop:compact:radius}, then the radius $\varepsilon$ of the Wasserstein ball that encloses the hyperrectangles is guaranteed to be strictly smaller than the radius $\widehat\varepsilon$ of the monolithic ball when 
\begin{align*}
N\ge\big(cn^{1/p+\max\{0,1/q-1/p\}}\big)^{1/d_{\max}-1/d}
\end{align*}
and decreases much faster for larger $N$, where $N^{-\frac{1}{d_{\rm max}}}\ll N^{-\frac{1}{d}}$  (cf. Figure~\ref{fig:rectangles:vs:balls}). The  center $\bm P_\xi^N$ of the ball enclosing the hyperrectangles is different from the center $P_\xi^N$ of the monolithic ball, since $P_\xi^N$ is the empirical distribution $\frac{1}{N}\sum_{i=1}^N\delta_{\xi^i}$, whereas $\bm{P}_\xi^N$ is the product empirical distribution
\begin{align*}
P_{\xi_1}^N\otimes\cdots\otimes P_{\xi_n}^N=\frac{1}{N^n}\sum_{(i_1,\ldots,i_n)\in[N]^n}\delta_{(\xi_1^{i_1},\ldots,\xi_n^{i_n})}.
\end{align*}
\begin{figure}[h]
    \centering
    \hspace*{-2.8cm}
    \includegraphics[width=1.33\linewidth]{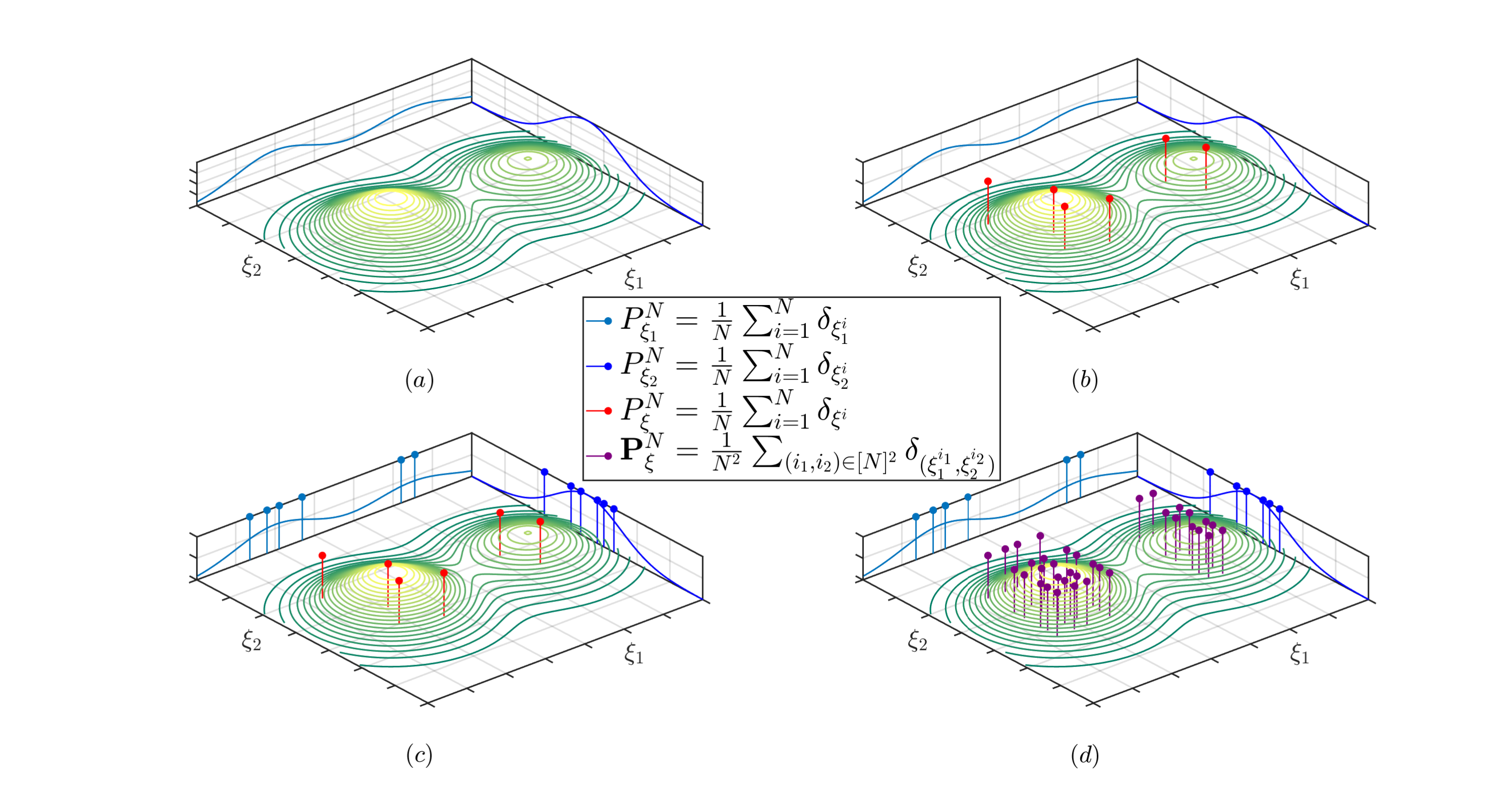}
    \caption{The figure illustrates an example of the reference  distributions $P_\xi^N$ and $\bm P_\xi^N$. (a) shows a contour plot of the product distribution $P_\xi=P_{\xi_1}\otimes P_{\xi_2}$, together with its marginals $P_{\xi_1}$ and $P_{\xi_2}$. (b) depicts the empirical distribution $P_\xi^N$ in red, which is formed by taking six samples from the true distribution $P_\xi$. (c) shows the marginals $P_{\xi_1}^N$ and $P_{\xi_2}^N$ of the empirical distribution $P_\xi^N$ in light-blue and blue, respectively.  Finally, (d) illustrates the product empirical distribution $\bm P_\xi^N=P_{\xi_1}^N\otimes P_{\xi_2}^N$ in purple, which is formed by taking the product of the marginal empirical distributions. The product empirical distribution $\bm P_\xi^N$ is clearly an improved approximation of the true distribution compared to the empirical distribution $P_\xi^N$.}
    \label{fig:virtual:samples}
\end{figure}
(cf. Figures~\ref{fig:rectangles:vs:balls},~\ref{fig:virtual:samples}). This also implies that under Assumption~\ref{assumption:independence}, an ambiguity ball that is centered at the product empirical distribution $\bm P_\xi^N$ will contain the true distribution with significantly higher probability compared to when it is centered at the empirical distribution $P_\xi^N$. 

The favorable decay rate of the ambiguity rectangles is further justified by the fact that the corresponding radius of monolithic Wasserstein balls can in principle not be improved, besides potentially a constant factor that is independent of the samples.  
Indeed, for any distribution $P_\xi\in\mathcal P_p(\Xi)$ for which ${\rm supp}(P_\xi)$ has a non-empty interior in $\Rat{d}$, 
\begin{align}
\widehat\eps_\star:=\Big( \frac{d}{d+p}\Big)^{1/p} C_\star^{-1/d} \frac{1}{N^{1/d}} \le W_p(P_\xi^N,P_\xi) \label{eq:lower:bound1}
\end{align}
always holds for some $C_\star>0$, since the lower bound $\widehat\eps_\star$ in \eqref{eq:lower:bound1} holds for any discrete distribution in place of $P_\xi^N$ that is supported on $N$ points (cf.~\cite[Proposition 4.2]{BK:12}). Namely, to contain the distribution $P_\xi$ with nonzero probability, the monolithic ambiguity ball centered at $P_\xi^N$ needs to have a radius at least $\widehat\eps_\star$, as shown with the dashed circle in Figure~\ref{fig:rectangles:vs:balls}, which shrinks at the same rate $N^{-1/d}$ as the radius $\widehat\eps$ in \eqref{explicit:ambiguity:radius}.

\begin{rem} \longthmtitle{Tightness of hyperrectangle bounds} 
When the lower-dimensional components of $\xi$ have the same dimension, i.e., $d_k=d/n=d_{\rm max}$ for all $k\in[n]$, the  radius $\eps$ of the Wasserstein ball enclosing the ambiguity hyperrectangles decays optimally. Here optimality is interpreted in the sense that  $\eps$ has at least the same decay rate as the Wasserstein distance between the true distribution and any discrete distribution with the same number of points as the product empirical distribution (a faster decay would otherwise imply that the enclosing ball will eventually contain the true distribution with zero probability). Indeed, assuming again that ${\rm supp}(P_\xi)$ has a non-empty interior, we get from \cite[Proposition 4.2]{BK:12} that the bound
\begin{align*}
  \Big( \frac{d}{d+p}\Big)^{1/p} C_\star^{-1/d} \frac{1}{N^{n/d}} \le W_p(Q,P_\xi)
\end{align*}
always holds for any discrete distribution $Q$ that is supported on $N^n$ points. Then the conclusion follows from \eqref{epsilon:decay:rectangles} and the fact that $N^{1/d_{\max}}=N^{n/d}$. 
\end{rem}

\section{DRO reformulations over ambiguity hyperrectangles}
\label{sec:duality}

In this section, we provide dual reformulations for the DRO problem \eqref{DR problem} when the ambiguity set $\mathcal P^N$ is the Wasserstein hyperrectangle in \eqref{hyperrectangle}, or the multi-transport hyperrectangle in \eqref{transport hyperrectangle}. Namely, we provide equivalent forms of the problem, which avoid the maximization over the space of probability distributions and are a stepping stone to obtain tractable optimization algorithms. We are therefore interested in reformulating the inner maximization problem   
\begin{align*}
     \sup_{P\in\mathcal P^N}
     \mathbb{E}_{P}[h(\xi)],
\end{align*}
where we fix the decision variable $x$ in \eqref{DR problem} and denote $h(\xi):=f(x,\xi)$ to facilitate notation. We next provide dual reformulations of these problems, first for Wasserstein hyperrectangles and then for multi-transport hyperrectangles. The former reformulations are applicable to a narrower class of objective functions because  Wasserstein hyperrectangles are non-convex. Nevertheless, these reformulations provide sharper results since Wasserstein hyperrectangles are typically strictly contained inside multi-transport hyperrectangles. 

\subsection{Dual reformulations over Wasserstein hyperrectangles}
\label{subsec:duality:Wasserstein:hyperrectangles}

Here we provide dual reformulations of the DRO problem \eqref{DR problem} when the set $\Xi$ where the distribution is supported has the product-structure of Assumption \ref{assumption:independence}(i) and the ambiguity set $\mathcal P^N$ is the Wasserstein hyperrectangle $\mathcal H_p(Q,\bm\varepsilon)$ for some product distribution $Q=Q_1\times\cdots\times Q_n$. Thus, we are interested to determine the dual of the inner problem
\begin{align}
     \sup_{P\in\mathcal H_p(Q,\bm\varepsilon)}
     \mathbb{E}_{P}[h(\xi)].
     \label{inner:maximization0}
\end{align}
To this end, we assume that $h$ can be written as the sum or product of functions that depend only on the individual components of $\xi$ and are integrable with respect to the corresponding marginals of the reference distribution.

\begin{assumption}
\longthmtitle{Sum/product decomposition}
(i) The objective function $h$ can be expressed as the sum of upper semicontinuous functions or the product of nonnegative upper semicontinuous functions that depend only on the respective components of the random variable. Namely, 
\begin{subequations}
\begin{align}
h(\xi) & =\sum_{k=1}^n h_k(\xi_k)  \label{assumption:sum} \\
{\rm or}\quad h(\xi) & = \prod_{k=1}^n h_k(\xi_k),\qquad h_k(\xi_k)\ge 0. 
\label{assumption:prod}
\end{align}
\end{subequations} 
\label{assumption f sum prod}
(ii) Each function $h_k$ is integrable with respect to $Q_k$.  
\end{assumption}
We will use the following strong duality result for the maximization over Wasserstein balls.
\begin{prop}\label{lemma dual}
\longthmtitle{DRO dual over Wasserstein balls \cite[Theorem 1]{JB-KM:19}}
Given a Polish space $\Xi$, consider the Wasserstein ball $\mathcal{B}_p(Q,\varepsilon)$ with $Q\in\mathcal P_p(\Xi)$ and the upper semicontinuous function $h\in L^1(Q)$. Then 
\begin{align*}
\sup_{P\in\mathcal{B}_p(Q,\varepsilon)}\mathbb{E}_{P}[h(\xi)] =\inf_{\lambda\ge 0} \int_{\Xi}\sup_{\xi\in\Xi}\{h(\xi)+\lambda (\varepsilon^p-\rho(\zeta,\xi)^p)\}dQ(\zeta).
\end{align*}
\end{prop}

The following result establishes strong duality for DRO problems with Wasserstein hyperrectangles when the objective function satisfies Assumption \ref{assumption f sum prod}.
\begin{prop} \label{prop:hyperrectangle:reformulation}
\longthmtitle{DRO dual over Wasserstein hyperrectangles}
Let the objective function $h$ of \eqref{inner:maximization0} satisfy Assumption \ref{assumption f sum prod}. Then \eqref{inner:maximization0} admits the duals \begin{subequations}
\begin{align}
   &\inf_{\bm\lambda\succeq0} \sum_{k=1}^n\int_{\Xi_k}\sup_{\xi\in\Xi_k}\{h_k(\xi_k)+\lambda_k(\varepsilon_k^p-\rho_k(\zeta_k,\xi_k)^p)\}dQ_k(\zeta_k)\label{dual sum}\\
    &\inf_{\bm\lambda\succeq0}\prod_{k=1}^n\int_{\Xi_k}\sup_{\xi\in\Xi_k}\{h_k(\xi_k)+\lambda_k (\varepsilon_k^p-\rho_k(\zeta_k,\xi_k)^p)\}dQ_k(\zeta_k),
    \label{dual prod}
\end{align}
\label{eq:sum/prod dual}
\end{subequations}
corresponding to Assumptions \eqref{assumption:sum} and \eqref{assumption:prod}, respectively, where $\bm\lambda=(\lambda_1,\ldots,\lambda_n)$.
\end{prop}
\begin{proof}
The derivation of \eqref{dual sum} follows from the fact that under \eqref{assumption:sum},
\begin{align*}
    \sup_{P\in\mathcal{H}_p(Q,\bm\varepsilon)}\mathbb{E}_{P}\big[h(\xi) \big] &\overset{(a_\Sigma)}{=}\sup_{P\in\mathcal{H}_p(Q,\bm\varepsilon)}\mathbb{E}_{P}\bigg[\sum_{k=1}^nh_k(\xi_k) \bigg] \\
     & \overset{(b_\Sigma)}{=} \sup_{P_{k}\in\mathcal{B}_p(Q_k,\varepsilon_k),k\in[n]} 
    \mathbb{E}_{P_{1}\otimes\cdots\otimes P_{n}}\bigg[\sum_{k=1}^nh_k(\xi_k) \bigg]  \\
    & \overset{(c_\Sigma)}{=}\underset{P_{k}\in\mathcal{B}_p(Q_k,\varepsilon_k),k\in[n]}{\mathrm{sup}}\sum_{k=1}^n\mathbb{E}_{P_{1}\otimes\cdots\otimes P_{n}}\big[h_k(\xi_k) \big] \\
    & \overset{(d_\Sigma)}{=}\underset{P_{k}\in\mathcal{B}_p(Q_k,\varepsilon_k),k\in[n]}{\mathrm{sup}}\sum_{k=1}^n\mathbb{E}_{P_{k}}\big[h_k(\xi_k) \big]\\
    &\overset{(e_\Sigma)}{=}\sum_{k=1}^n\inf_{\lambda_k\ge0}\int_{\Xi_k}\sup_{\xi\in\Xi_k}\{h_k(\xi_k)+\lambda_k (\varepsilon_k^p-\rho_k(\zeta_k,\xi_k)^p)\}dQ_k(\zeta_k)\\
    &\overset{(f_\Sigma)}{=}\inf_{\bm\lambda\succeq0}\sum_{k=1}^n\int_{\Xi_k}\sup_{\xi\in\Xi_k}\{h_k(\xi_k)+\lambda_k (\varepsilon_{k}^p-\rho_k(\zeta_k,\xi_k)^p)\}dQ_k(\zeta_k).
\end{align*}
Here,  $(a_\Sigma)$ is a consequence of Assumption \ref{assumption f sum prod} and $(b_\Sigma)$ follows from the definition of the Wasserstein hyperrectangle in \eqref{hyperrectangle}. Linearity of the expectation yields $(c_\Sigma)$ and $(d_\Sigma)$ follows by exploiting Fubini's theorem (cf.~\cite[Theorem 14.19]{AK:13}). To derive $(e_\Sigma)$ we used Proposition \ref{lemma dual} and $(f_\Sigma)$ follows from the fact that $\sum_k  \inf_{\lambda_k\ge0}\psi_k(\lambda_k)=\inf_{\bm\lambda\succeq 0}\sum_k\psi_k(\lambda_k)$ for any functions $\psi_k$.
Thus, \eqref{dual sum} holds.

In a similar manner, under \eqref{assumption:prod}, we have  
\begin{align*}
\sup_{P\in\mathcal{H}_p(Q,\bm\varepsilon)}\mathbb{E}_{P}\big[h(\xi) \big] 
& \overset{(a_\Pi)}{=}\underset{P_{k}\in\mathcal{B}_p(Q_k,\varepsilon_k),k\in[n]}\sup\mathbb E_{P_{1}}\Big[\cdots\mathbb E_{P_{n}}\Big[\prod_{k=1}^nh_k(\xi_k)\Big]\cdots\Big]\\
    & \overset{(b_\Pi)}{=}  \underset{P_{k}\in\mathcal{B}_p(Q_k,\varepsilon_k),k\in[n]}{\mathrm{sup}}\prod_{k=1}^n\mathbb{E}_{P_{k}}\big[h_k(\xi_k) \big]\\
    &\overset{(c_\Pi)}{=} \prod_{k=1}^n\inf_{\lambda_k\ge0}\int_{\Xi_k}\sup_{\xi\in\Xi_k}\{h_k(\xi_k)+\lambda_k (\varepsilon_{k}^p-\rho_k(\zeta_k,\xi_k)^p)\}dQ_k(\zeta_k) \\
    &\overset{(d_\Pi)}{=}\inf_{\bm\lambda\succeq0}\prod_{k=1}^n\int_{\Xi_k}\sup_{\xi\in\Xi_k}\{h_k(\xi_k) +\lambda_k (\varepsilon_{k}^p-\rho_k(\zeta_k,\xi_k)^p)\}dQ_k(\zeta_k), 
\end{align*}
namely, \eqref{dual prod} holds. In these derivations, $(a_\Pi)$ follows from Assumption \ref{assumption f sum prod}, \eqref{hyperrectangle}, and  Fubini's theorem (cf. \cite[Theorem 14.19]{AK:13}), and $(b_\Pi)$ from linearity of the expectation. Furthermore, $(c_\Pi)$ follows from Proposition \ref{lemma dual} and $(d_\Pi)$ from the fact that $\prod_k \inf_{\lambda_k\ge0}\psi_k(\lambda_k)=\inf_{\bm\lambda\succeq0}\prod_k\psi_k(\lambda_k)$ for any nonnegative functions $\psi_k$. The proof is now complete. 
\end{proof}

The following corollary provides the dual reformulation of Proposition \ref{lemma dual} for the case when the center of the Wasserstein hyperrectangle is the product empirical distribution.

\begin{corollary}
\longthmtitle{Dual of data-driven Wasserstein  hyperrectangles}
Let $h$ satisfy Assumption~\ref{assumption:h}(i). Then \eqref{inner:maximization0} with $Q\equiv\bm P_\xi^N$ admits the corresponding duals
\begin{align*}
   &\inf_{\bm\lambda\succeq0} \sum_{k=1}^n\frac{1}{N}\sum_{i=1}^N\underset{\xi_k\in\Xi_k}{\mathrm{sup}}\{h_k(\xi_k)+\lambda_k(\varepsilon_k^p-\rho_k(\xi_k^i,\xi_k)^p)\} \\
    &\inf_{\bm\lambda\succeq0}\prod_{k=1}^n\frac{1}{N}\sum_{i=1}^N\underset{\xi_k\in\Xi_k}{\mathrm{sup}}\{h_k(\xi_k)+\lambda_k (\varepsilon_k^p-\rho_k(\xi_k^i,\xi_k)^p) \}. 
 \end{align*}
\end{corollary}

\subsection{Dual reformulations over multi-transport hyperrectangles}
\label{subsec:duality:multi:transport}

Here, we provide the dual of the inner maximization problem in \eqref{DR problem} when the ambiguity set $\mathcal P^N$ is the multi-transport hyperrectangle \eqref{transport hyperrectangle}. Namely, we reformulate the problem 
\begin{align}
     \sup_{P\in\mathcal T(Q,\bm\epsilon)}
     \mathbb{E}_{P}[h(\xi)].
     \label{inner:maximization}
\end{align}

To obtain the dual of \eqref{inner:maximization}, we depart from the necessity of Section~\ref{subsec:duality:Wasserstein:hyperrectangles} to have a Polish space with a product structure and assume that the uncertainty $\xi$ belongs to a general Polish space $\Xi$.
Our analysis generalizes the duality approach in \cite{JB-KM:19}, which obtains dual reformulations of DRO problems where distributional ambiguity is captured through a single optimal transport constraint. As in \cite{JB-KM:19}, we make the following assumption for $h$.
\begin{assumption}
\label{assumption:h}
\longthmtitle{Objective function class}
The objective function $h:\Xi\to\Rat{}$ 
is upper semicontinuous and $h\in L^1(Q)$.
\end{assumption}

We also assume the following regarding the cost functions in \eqref{transport:plan:set}.

\begin{assumption} 
\longthmtitle{Transport costs}
(i) For each $k\in[n]$ there exists a nondecreasing sequence $c_{k,m}:\Xi\times \Xi\to\Rat{}_{\ge 0}$, $m\in\mathbb N$ of \textit{continuous} transport costs with $c_{k,m}(\zeta,\zeta)=0$ for all $\zeta$ and $c_{k,m}(\zeta,\xi)\nearrow c_k(\zeta,\xi)\in \Rat{}_{\ge 0}$. \\
(ii)  There exists a compact set $\Xi_{\rm cmp}\subset \Xi$ such that for each $m$, $c_{k,m}$, $k\in[n]$ are linearly independent in $C(\Xi_{\rm cmp}\times\Xi_{\rm cmp})$ and $\text{span}\{c_{1,m},\ldots,c_{n,m}\}\cap C_{2,{\rm const}}(\Xi_{\rm cmp}\times\Xi_{\rm cmp})=\{0\}$. 
\label{assumption:cost:functions}
\end{assumption}

Assumption \ref{assumption:cost:functions} is directly satisfied when $\Xi$ has the product structure of Assumption~\ref{assumption:independence}   and the considered cost functions are powers of the distances between the components of the random variable, i.e., when  $c_k(\zeta,\xi)=\rho_k(\zeta_k,\xi_k)^p$. We next provide some preparatory definitions and sketch the intuition behind the strong dual to problem \eqref{inner:maximization}, which is given later in this section. 

Let $\mathcal I(\pi):=\int_{\Xi\time\Xi}h(\xi)d\pi(\zeta,\xi)$ and consider the set $\P^h(\Xi)$ of distributions $\nu$ on $\Xi$ for which the integral of $h$ is well defined and takes values in $\bar{\mathbb R}$, namely, for which either $\int_\Xi h_+(\xi)d\nu(\xi)\in\Rat{}$ or $\int_{\Xi}h_-(\xi)d\nu(\xi)\in\Rat{}$, were $h_+:=\max\{h,0\}$ and $h_-:=\min\{h,0\}$. Denoting further $\Pi^h(Q):=\{\pi\in\Pi(Q):{\rm pr}_{2\#}\pi\in\P^h(\Xi)\}$ and analogously $\Pi^h(Q,\bm\epsilon)$, and recalling the definition of $\mathcal T_p(Q,\bm\epsilon)$, allows us to rigorously (re)define\footnote{The same argument can be used to resolve potential ambiguities in the reformulations of Section~\ref{subsec:duality:Wasserstein:hyperrectangles} when there are distributions in the ambiguity set that may lead to integrals of the form $+\infty-\infty$. An alternative  way to address this issue is to define $+\infty-\infty=+\infty$ as in 
\cite{JB-KM:19}. Endnote 2 in \cite{JB-KM:19} also clarifies why such ambiguities do not affect the interpretation of the optimization problem.} the DRO problem \eqref{inner:maximization} as  
\begin{align}
\sup_{P\in\mathcal T_p(Q,\bm\epsilon)\cap\mathcal P^h(\Xi)}\int_{\Xi}h(\xi)dP(\xi) =\sup_{\pi\in \Pi^h(Q,\bm\epsilon)}\int_{\Xi\time\Xi}h(\xi)d\pi(\zeta,\xi)=\sup_{\pi\in \Pi^h(Q,\bm\epsilon)}\mathcal I(\pi)=:\mathcal I^\star. \label{inner:maximization:reformulation}
\end{align}
Due to \eqref{transport:plan:set}, this is a linear optimization problem in the space of finite signed measures on $\Xi\times\Xi$. When restricted further over the convex set of probability measures 
\begin{align*}
\Pi_{{\rm fin},\bm c}(Q):=\Big\{\pi\in\Pi^h(Q):\int_{\Xi\time\Xi}c_k(\zeta,\xi)d\pi(\zeta,\xi)<+\infty\;\textup{for all}\;k\in[n]\Big\},
\end{align*}
over which  the integrals of the costs are real-valued, and taking into account the inequality constraints \eqref{transport:plan:set}, which always imply $\Pi^h(Q,\bm\epsilon)\subset\Pi_{{\rm fin},\bm c}(Q)$, its Lagrangian is given by 
\begin{align}
    \mathcal L(\pi,\boldsymbol{\lambda}) & :=\int_{\Xi\times\Xi} h(\xi) d\pi(\zeta,\xi)+\sum_{k=1}^n\lambda_k\Big(\epsilon_k-\int_{\Xi\times\Xi}c_k(\zeta,\xi)d\pi(\zeta,\xi)\Big) \nonumber \\
    & \phantom{:}=\langle\bm\lambda,\bm\epsilon\rangle+\int_{\Xi\times\Xi}(h(\xi)-\langle \bm\lambda,\bm c(\zeta,\xi) \rangle)d\pi(\zeta,\xi), \label{Lagrangian}
\end{align}
where $\boldsymbol{\lambda}=(\lambda_1,\ldots,\lambda_n)\in \Rat{n}_{\ge0}$, $\bm\epsilon=(\epsilon_1,\ldots,\epsilon_n)$, and $\bm c(\zeta,\xi)=(c_1(\zeta,\xi),\ldots,c_n(\zeta,\xi))$. From the definition of $\mathcal L(\pi,\bm\lambda)$ we have 
\begin{align}
\mathcal I^\star  =\sup_{\pi\in\Pi_{{\rm fin},\bm c}(Q)}\inf_{\bm\lambda\succeq0}\mathcal L(\pi,\bm\lambda) 
&= \sup_{\pi\in\Pi_{{\rm fin},\bm c}(Q)}\inf_{\bm\lambda\succeq0} \Big\{\langle\boldsymbol{\lambda, \epsilon}\rangle+\int_{\Xi\times\Xi}(h(\xi)-\langle \bm\lambda,\bm c(\zeta,\xi) \rangle)d\pi(\zeta,\xi)\Big\} \label{I:star}
\end{align}
and we get from the min--max inequality that 
\begin{align}
\mathcal I^\star\le \inf_{\bm\lambda\succeq0}\Big\{\langle\bm\lambda, \bm\epsilon\rangle+\sup_{\pi\in\Pi_{{\rm fin},\bm c}(Q)}\int_{\Xi\times\Xi}(h(\xi)-\langle \bm\lambda,\bm c(\zeta,\xi) \rangle)d\pi(\zeta,\xi)\Big\}.\label{I:star:inequality}
\end{align}

To provide some intuition behind the dual problem to \eqref{inner:maximization}, assume for the moment that $\Xi$ is compact and that $h$ and the cost functions $c_1,\ldots,c_n$ are continuous. In this case we have that $\Pi_{{\rm fin},\bm c}(Q)=\Pi(Q)$ and the maximization problem
\begin{align*}
\sup_{\pi\in \Pi(Q)}\int_{\Xi\times\Xi}(h(\xi)-\langle \bm\lambda,\bm c(\zeta,\xi) \rangle)d\pi(\zeta,\xi)
\end{align*}
is a linear program that can be written in the abstract form  
\begin{align*}
\sup & \langle e,\pi \rangle_1 \\
{\rm s.t.}\; & \mathcal A\pi=b \\
& \pi \succeq 0.
\end{align*}
Here $\mathcal A\equiv {\rm pr}_{1\#}:\mathcal M(\Xi\times\Xi)\to\mathcal M(\Xi)$, $b\equiv Q\in\mathcal M(\Xi)$, $e\equiv h\circ{\rm pr}_2-\sum_{k=1}^n\lambda_k c_k\in C(\Xi\times\Xi)$, $\langle\cdot,\cdot\rangle_1$ denotes the duality between $C(\Xi\times\Xi)$ and $\mathcal M(\Xi\times\Xi)$, and the order $\succeq$ is taken with respect to the cone of positive measures on $\Xi\times\Xi$. By linear programming duality (cf.~\cite{AB:02}), its  dual problem is given by  
\begin{align*}
\inf & \langle \varphi,b \rangle_2 \\
{\rm s.t.}\; & \mathcal A^*\varphi\succeq e. 
\end{align*} 
Here $\mathcal A^*\equiv \mathcal K_{{\rm pr}_1}:C(\Xi)\to C(\Xi\times\Xi)$ is the adjoint of $\mathcal A$, namely the composition, a.k.a. Koopman operator (cf.~\cite[Chapter 4.3]{TE-BF-MH-RN:15}), with $\mathcal K_{{\rm pr}_1}(\varphi):=\varphi\circ{\rm pr}_1$,  $\langle\cdot,\cdot\rangle_2$ denotes the duality between  $C(\Xi)$ and $\mathcal M(\Xi)$, and the order $\succeq$ is taken with respect to the cone of positive continuous functions on $\Xi$. This in turn is an abstract representation of the dual problem   
\begin{align*}
\inf\Big\{\int_\Xi\varphi(\zeta)dQ(\zeta): \varphi\in C(\Xi)\;{\rm and}\; \varphi(\zeta)\ge h(\xi)-\langle \bm\lambda,\bm c(\zeta,\xi)\rangle \;\textup{for all}\;\zeta,\xi\in\Xi\Big\}. 
\end{align*} 

Based on these considerations, we introduce the dual of \eqref{inner:maximization:reformulation} in the general case, where $\Xi$ does not need to be compact and $h$, $c_k$ are not necessarily continuous. To this end, we denote  
\begin{subequations} \label{Lambda:and:J}
\begin{align}
\Lambda \equiv \Lambda(h;c_1,\ldots,c_n) 
& := \Big\{(\bm \lambda,\varphi): \bm \lambda\succeq 0,\varphi\in \mathfrak m_{\mathcal U}(\Xi;\Rat{}\cup\{+\infty\}) \nonumber \\
& \hspace{8.5em} {\rm and}\; \varphi\circ{\rm pr}_1\succeq  h\circ{\rm pr}_2-\sum_{k=1}^n\lambda_k c_k\Big\} \\
\mathcal J(\bm\lambda,\varphi) &  :=\langle\bm\lambda,\bm\epsilon\rangle+\int_\Xi\varphi(\zeta)dQ(\zeta)
\end{align}
\end{subequations}
and consider in analogy to \cite{JB-KM:19}\footnote{From  Assumptions~\ref{assumption:h} and \ref{assumption:cost:functions}(i), it follows that for any function $\varphi\in \mathfrak m_{\mathcal U}(\Xi;\Rat{}\cup\{+\infty\})$ the integral $\int_\Xi\varphi(\zeta)dQ(\zeta)$ is well defined. This ensures that the integral of the function $\varphi_{\bm\lambda}(\zeta):=\sup_{\xi\in\Xi}\{h(\xi)-\langle \bm\lambda,\bm c(\zeta,\xi)\rangle\}$ in the attainable dual pair in Theorem~\ref{thm:strong:duality} is also well defined, since $\varphi_{\bm\lambda}$ is universally measurable (cf. \cite[Page 16]{JB-KM:19} for the justification of this fact). Measurability of the integrands in the dual reformulations of Section~\ref{subsec:duality:Wasserstein:hyperrectangles} is guaranteed in the same way.} the dual problem
\begin{align}
\mathcal J_\star:=\inf_{(\bm\lambda,\varphi)\in\Lambda}\mathcal J(\bm\lambda,\varphi)=
\inf_{\bm\lambda\succeq0}\Big\{\langle\bm\lambda,\bm\epsilon\rangle & +\inf\Big\{\int_\Xi\varphi(\zeta)dQ(\zeta): \varphi\in \mathfrak m_{\mathcal U}(\Xi;\Rat{}\cup\{+\infty\}) \nonumber \\
& {\rm and}\; \varphi(\zeta)\ge h(\xi)-\langle \bm\lambda,\bm c(\zeta,\xi)\rangle \;\textup{for all}\;\zeta,\xi\in\Xi\Big\}\Big\}. \label{J:star}
\end{align}
Then it follows from \eqref{I:star:inequality} that  
\begin{align} \label{weak:duality}
\mathcal I^\star\le\mathcal J_\star. 
\end{align}

The establishment of strong duality between the primal optimization problem and its dual hinges on showing that the reverse inequality also holds. Its proof is given in Appendix~\ref{appendix:duality:proof} and it is based on appropriate modifications of the technical approach developed in \cite{JB-KM:19}. 

\begin{thm}
\label{thm:strong:duality}
\longthmtitle{DRO dual over multi-transport hyperrectangles}
Consider the problem \eqref{inner:maximization} and let $h$ and $c_1,\ldots,c_n$ satisfy Assumptions \ref{assumption:h} and \ref{assumption:cost:functions}, respectively. Then 
\begin{align}
   \mathcal I^\star=\mathcal J_\star=\inf_{\bm\lambda\succeq0}\Big\{\langle\bm\lambda, \bm\epsilon\rangle+\int_{\Xi}\sup_{\xi\in\Xi}\{h(\xi)-\langle \bm\lambda,\bm c(\zeta,\xi) \rangle \}dQ(\zeta)\Big\} \label{eq:strong duality}
\end{align}
and there exist $(\bm\lambda,\varphi_{\bm\lambda})\in\Lambda$ with $\varphi_{\bm\lambda}(\zeta):=\sup_{\xi\in\Xi}\{h(\xi)-\langle \bm\lambda,\bm c(\zeta,\xi) \rangle \}$, $\zeta\in\Xi$, for which the infimum in \eqref{J:star} is attained.
\end{thm}

Using this result we obtain the following explicit reformulation of the DRO problem when $\Xi$ has a product structure and the ambiguity set is a data-driven multi-transport hyperrectangle.

\begin{corollary}
\label{cor:dual}  
\longthmtitle{Dual of data-driven multi-transport hyperrectangles}
If $\Xi$ satisfies Assumption~\ref{assumption:independence}(i), then the optimal value of \eqref{inner:maximization} with $\mathcal T(Q,\bm\epsilon)\equiv\mathcal T_p(\bm P_\xi^N,\bm\eps)$  is equal to  
\begin{align}
 \mathcal I^\star=\inf_{\bm\lambda\succeq0}\Big\{\langle\bm\lambda,\bm\epsilon\rangle+\frac{1}{N^n}\sum_{(i_1,\ldots,i_n)\in[N]^n}\sup_{\xi\in\Xi}\Big\{h(\xi)-\sum_{k=1}^n \lambda_k \rho_k(\xi_{k}^{i_k},\xi_k)^p\Big\}\Big\}. 
 \label{eq:duality2}
 \end{align}
\end{corollary}

\begin{rem}
\longthmtitle{Strict improvement of Wasserstein hyperrectangle optimal value}
Despite the fact that multi-transport hyperrectangles admit dual reformulations over a much broader class of objective functions compared to Wasserstein hyperrectangles, the latter can exhibit a strict improvement of their optimal values compared to the former. This is justified by the containment result of Proposition \ref{prop:containment} and is illustrated in the following toy example. 

Consider the product reference distribution 
\begin{align*}
Q=Q_1\otimes Q_2:= \; & (p_1\delta_0+(1-p_1)\delta_1)\otimes(p_2\delta_0+(1-p_2)\delta_1) \\
= \; & p_1p_2\delta_{(0,0)}+(1-p_1)p_2\delta_{(1,0)}+p_1(1-p_2)\delta_{(0,1)}+(1-p_1)(1-p_2)\delta_{(1,1)}
\end{align*}
on $\Rat{2}$, the objective function
\begin{align*}
h(\xi):=\mathds 1_{\{(0,0)\}}(\xi)\equiv \mathds 1_{\{0\}}(\xi_1)\mathds 1_{\{0\}}(\xi_2),
\end{align*} 
and ambiguity radii $\eps_1\le 1-p_1$ and $\eps_2\le 1-p_2$. 
For each of the ambiguity sets $\mathcal H_p(Q,\bm\eps)$ and $\mathcal T_p(Q,\bm\eps)$, the distribution that maximizes the value of $h$ is the one obtained when the largest possible amount of mass is transferred from the reference distribution to the point $(0,0)$. 

For the Wasserstein hyperrectangle, this distribution is obtained through the transport plans $\pi_1$ and $\pi_2$, which move the largest possible amount of mass from $1$ to $0$ using the transport budgets $\eps_1$ and $\eps_2$, respectively. Identifying these transport plans with their restrictions to $\{0,1\}^2\subset\Rat{2}$ where they are supported, we obtain their  matrix representations 
\begin{align*}
\pi_1\equiv \underbrace{\left(\begin{matrix}
p_1 &  \eps_1 \\
0 & 1-p_1-\eps_1 
\end{matrix}\right)}_{Q_1}
\bigg\}
{\scriptsize 
\begin{matrix}
P_1
\end{matrix}}
\qquad{\rm and}\qquad 
\pi_2\equiv
\underbrace{\left(\begin{matrix}
p_2 & \eps_2  \\
0 & 1-p_2-\eps_2
\end{matrix}\right)}_{Q_2}
\bigg\}
{\scriptsize 
\begin{matrix}
P_2
\end{matrix}}.
\end{align*}
The column sums of these transport plan matrices correspond to the reference distributions $Q_1$ and $Q_2$ and their row sums to the other marginals $P_1:=(p_1+\eps_1)\delta_0+(1-p_1-\eps_1)\delta_1$ and $P_2:=(p_2+\eps_2)\delta_0+(1-p_2-\eps_1)\delta_1$ of the transport plans. The distribution from $\mathcal H_p(Q,\bm\eps)$ that maximizes $h$ is $P:=P_1\otimes P_2$, which is also the second marginal of the transport plan  $\pi=T_{\#}(\pi_1\otimes\pi_2)$ (its first marginal is $Q$), where $T(\zeta_1,\xi_1,\zeta_2,\xi_2):=(\zeta_1,\zeta_2,\xi_1,\xi_2)$  (see proof of Proposition~\ref{prop:containment}). As above, we identify the transport plan  $\pi$ with its restriction to $\{(0,0),(1,0),(0,1),(1,1)\}^2\subset\Rat{4}$ where it is supported. Using the lexicographical ordering
\begin{align*}
(0,0)\quad (1,0)\quad (0,1)\quad (1,1)\quad \mapsto\quad 1\quad 2\quad 3\quad 4 
\end{align*}
and the product expression $\pi(\zeta,\xi)=\pi_1(\zeta_1,\xi_1)\pi_2(\zeta_2,\xi_2)$ of  $\pi$, we get its matrix representation  
\begin{align*}
\pi\equiv
\underbrace{\left(\begin{matrix}
p_1p_2 & \eps_1p_2 & p_1\eps_2 & \eps_1\eps_2 \\
0 & (1-p_1-\eps_1)p_2 & 0 & (1-p_1-\eps_1)\eps_2 \\
0 & 0 & p_1(1-p_2-\eps_2) & \eps_1(1-p_2-\eps_2) \\
0 & 0 & 0 & (1-p_1-\eps_1)(1-p_2-\eps_2)
\end{matrix}\right)}_{Q}
\left.\begin{matrix}
\\
\\
\\
\\
\end{matrix}\right\rbrace
{\scriptsize 
\begin{matrix}
P
\end{matrix}}.
\end{align*}

For the multi-transport hyperrectangle, it is not hard to check that the distribution that maximizes the value of $h$ is obtained through the transport plan  $\pi'$, which uses the transport budgets $\eps_1$ and $\eps_2$ to move the largest possible amounts of mass from $(1,0)$ to $(0,0)$ and from $(0,1)$ to $(0,0)$, respectively. Thus, its corresponding matrix representation is 
\begin{align*}
\pi'\equiv
\underbrace{\left(\begin{matrix}
p_1p_2 & \eps_1 & \eps_2 & 0 \\
0 & (1-p_1)p_2-\eps_1 & 0 & 0 \\
0 & 0 & p_1(1-p_2)-\eps_2 & 0 \\
0 & 0 & 0 & (1-p_1)(1-p_2)
\end{matrix}\right)}_{Q}
\left.\begin{matrix}
\\
\\
\\
\\
\end{matrix}\right\rbrace
{\scriptsize 
\begin{matrix}
P'
\end{matrix}}.
\end{align*}

Assuming without loss of generality that $p_1<1$ and taking into account that  
\begin{align*}
\eps_1\le 1-p_1\quad{\rm and}\quad\eps_2\le 1-p_2, 
\end{align*}
it follows that the mass transported to $(0,0)$ with the Wasserstein hyperrectangle is strictly less than that with the multi-transport hyperrectangle, namely, 
\begin{align*}
\eps_1p_2+p_1\eps_2+\eps_1\eps_2=\eps_1p_2+\eps_2(p_1+\eps_1)\le\eps_1p_2+\eps_2(p_1+1-p_1)=\eps_1p_2+\eps_2<\eps_1+\eps_2. 
\end{align*}
Therefore, we get that  
\begin{align*}
\bE_P[h(\xi)]=p_1p_2+\eps_1p_2+p_1\eps_2+\eps_1\eps_2<p_1p_2+\eps_1+\eps_2=\bE_{P'}[h(\xi)],
\end{align*}
i.e., that the optimal value over the Wasserstein hyperrectangle is strictly below that of the multi-transport hyperrectangle (cf. Figure~\ref{fig:rectangles:DRO:comparison}).    

Since the distribution from $\mathcal T_p(Q,\bm\eps)$
with the largest amount of mass at $(0,0)$ has more mass at that point than any distribution from $\mathcal H_p(Q,\bm\eps)$, it follows that this distribution cannot belong to the convex hull of $\mathcal H_p(Q,\bm\eps)$. This is also why we depict the multi-transport hyperrectangle as a curved rectangle, which is strictly convex, instead of drawing a straight rectangle that would look like the convex hull of the Wasserstein hyperrectangle.  
\end{rem}

\begin{figure}[h]
	\centering 
	\includegraphics[width=.7\linewidth]{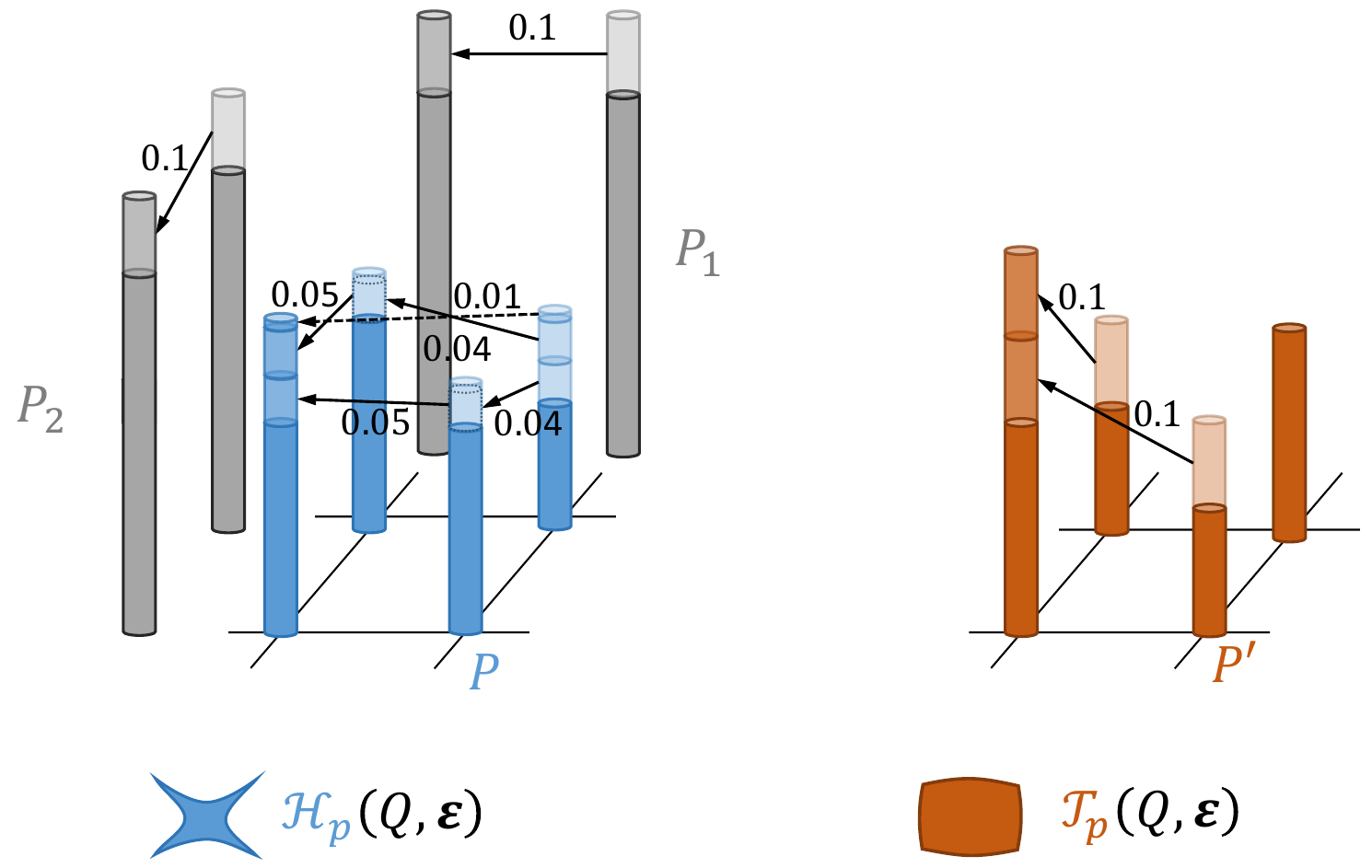} \\ 
	\caption{The figure illustrates the optimal transport plans corresponding to both ambiguity sets for the concrete values $p_1=p_2=0.5$ and $\eps_1=\eps_1=0.1$. Clearly, the mass transferred to $(0,0)$ to form the distribution $P$ that maximizes $h$ in the case of the Wasserstein hyperrectangle is considerably smaller than that used to form $P'$ in the case of the multi-transport hyperrectangle. This happens because $P$ needs to retain a product distribution structure, resulting in a redundant effort to transport mass from $(1,1)$ to all points $(0,1)$, $(1,0)$, and $(0,0)$, which is not required  to form $P'$.} \label{fig:rectangles:DRO:comparison}
\end{figure}

\section{Simulation example}
\label{sec:example}

In this section, we present an example where the duality results of the paper are used to solve the locational optimization problem of optimally placing a drone recharging station on a planar region. We consider two drones that are independently deployed in a large area that we identify with $\Rat{2}$ and can stop for recharging only when they are located in a subset $\Theta$ of this area. We assume that this subset is the unbounded region  $ \Theta:=\{\xi\in\Rat{2}: \xi\succeq 0 \}$ and that the recharging station can only be placed on   $\mathcal X:=\{x\in\Rat{2}:(0,0)\preceq x\preceq (5,5)\}$.  

The positions $\xi_1$ and $\xi_2$ of the drones are assumed random and independently distributed and we seek to minimize the expected sum of their quadratic distance from the recharging station when they are located in $\Theta$. This leads to the optimization problem  
\begin{align*}
    \inf_{x\in \mathcal X}\mathbb E_{P_\xi}[\mathds{1}_{\Theta\times\Theta}(\xi) ( \|x-\xi_1\|_2^2+ \|x-\xi_2\|_2^2)],
\end{align*}
where $\xi:=(\xi_1,\xi_2)$. The probability distribution  $P_\xi\in\P_2(\Rat{4})$ of the drones' positions is unknown and we only assume access to $N$ i.i.d. historic samples of it. Using these samples we build a data-driven ambiguity set $\mathcal P^N$ for $P_\xi$ and compare the solutions of the DRO problem 
\begin{align}
    \inf_{x\in \mathcal X} \sup_{P\in\mathcal P^N} \mathbb E_{P}[\mathds{1}_{\Theta\times\Theta}(\xi) (  \|x-\xi_1\|_2^2+ \|x-\xi_2\|_2^2)] 
    \label{eq:simEx}
\end{align}
when $\mathcal P^N$ is a multi-transport hyperrectangle and a Wasserstein ball. 
Using Corollary~\ref{cor:dual}, we can reformulate this problem with $\mathcal P^N\equiv\mathcal T_2(\bm P_\xi^N,\bm\eps)$ as   
\begin{align}
\left\{\begin{aligned}
\inf_{\substack{x\in\mathcal{X},\bm\lambda\succ\bm1 \\ \bm s\succeq0, \bm\nu\succeq0}}  & \langle \bm\lambda, \bm\epsilon \rangle + \frac{1}{N^2} \sum_{\bm i\in[N]^2} s_{\bm i} \\
 \text{s.t} & \sum_{k=1}^2\frac{\|r_k^{i_k}- \nu_k^{i_k}\|^2}{4(\lambda_k-1)}\le s_{\bm i}- \sum_{k=1}^2 (\| x\|^2-\lambda_k\|\xi_k^{i_k}\|^2),\quad \bm i\in[N]^2, 
 \end{aligned}\right.  \label{eq:opt:problem:ex:2}
\end{align}
where $\bm\epsilon:=(\eps_1^2,\eps_2^2)$, $\bm1:=(1,1)$, $\bm i:=(i_1,i_2)$, $\bm\nu:=\{(\nu_1^i,\nu_2^i)\}_{i\in[N]}$, and $\bm s:=\{\bm s_{\bm i}\}_{\bm i\in[N]^2}\in\Rat{N^2}$, with $\{\bm s_{\bm i}\}_{\bm i\in[N]^2}$ viewed as an element of $\Rat{N^2}$ for some ordering of $[N]^2$. Analogously, when $\mathcal P^N\equiv\mathcal{B}_2(P_\xi^N,\eps)$ we can reformulate the problem as  
\begin{align} 
\left\{\begin{aligned}\inf_{\substack{x\in\mathcal{X},\lambda>1 \\ \bm s\succeq0, \bm\nu\succeq 0}} & \lambda\eps^2 + \frac{1}{N} \sum_{i\in[N]} s_{i} \\
\text{s.t} & \sum_{k=1}^2\frac{\|r_k^i- \nu_k^i\|^2}{4(\lambda-1)}\le s_i- \sum_{k=1}^2 (\| x\|^2-\lambda\|\xi_k^i\|^2),\quad i\in[N],\end{aligned}\right. \label{eq:opt:problem:ex:3}
\end{align}
with $\bm s:=\{ s_i\}_{i\in[N]}\in\Rat{N}$ and $\bm\nu$ as above. The derivation of \eqref{eq:opt:problem:ex:2} and \eqref{eq:opt:problem:ex:3} is given in Appendix~\ref{appendix:simEx:derivations} where we also discuss how to reduce the complexity of problem~\eqref{eq:opt:problem:ex:2} by removing certain redundant constraints.

\begin{figure}
    \centering
    \includegraphics[width=.8\linewidth]{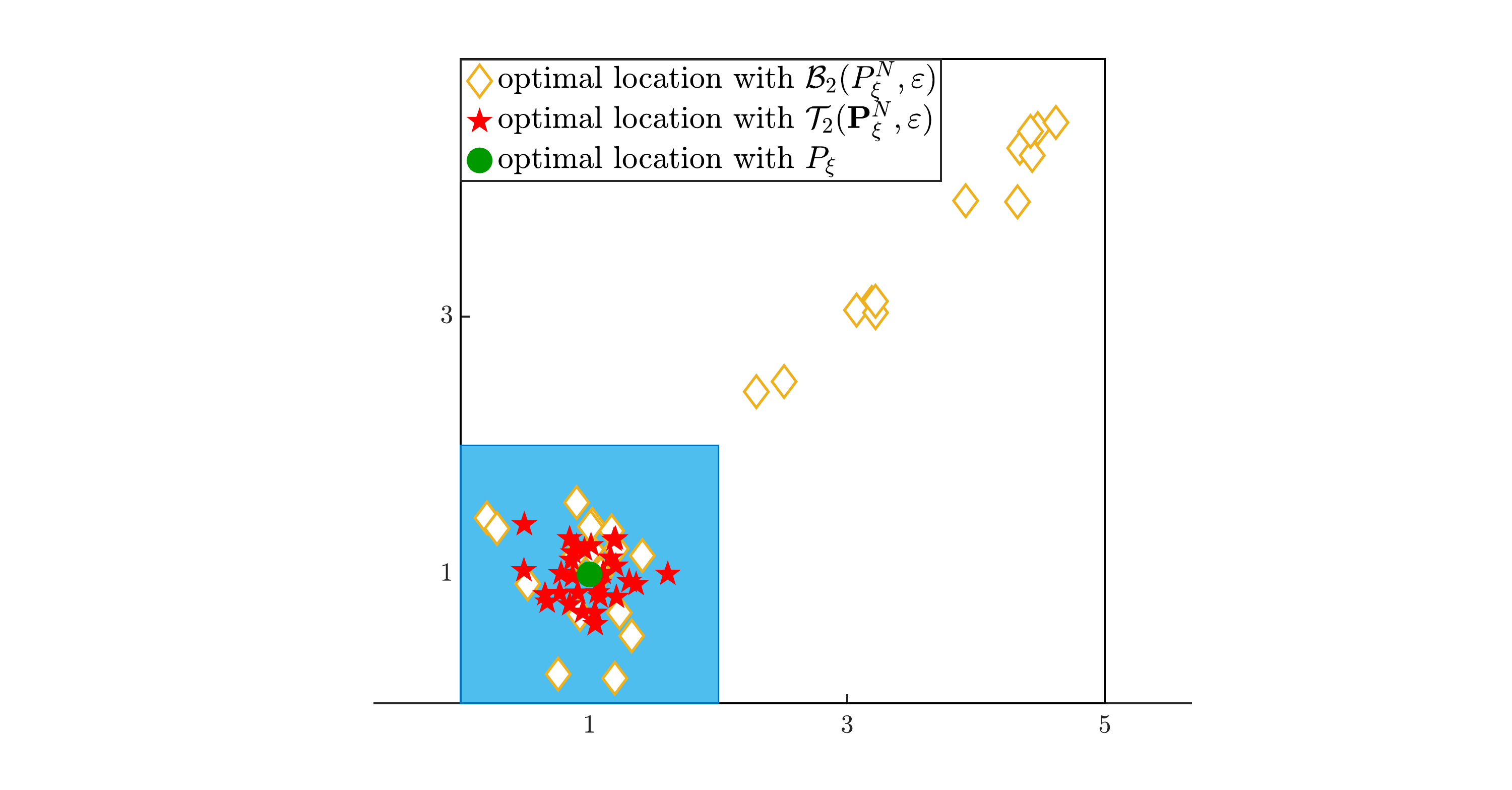}
    \caption{The figure shows the optimal recharging locations  across 30 realizations of the simulations. The blue rectangle represents the part of the recharging area where the drones may be located with nonzero probability and the green circle depicts their true optimal recharging location. The red stars and the yellow diamonds depict the optimal recharging locations that we obtained by solving the DRO problem using the multi-transport hyperrectangle and the  Wasserstein ball, respectively. It is evident that the locations obtained using the multi-transport hyperrectangle are on average much closer to the true optimal location.}
    \label{fig:RechargingLocations}
\end{figure}

\begin{figure}
    \centering
    \includegraphics[width=.9\linewidth]{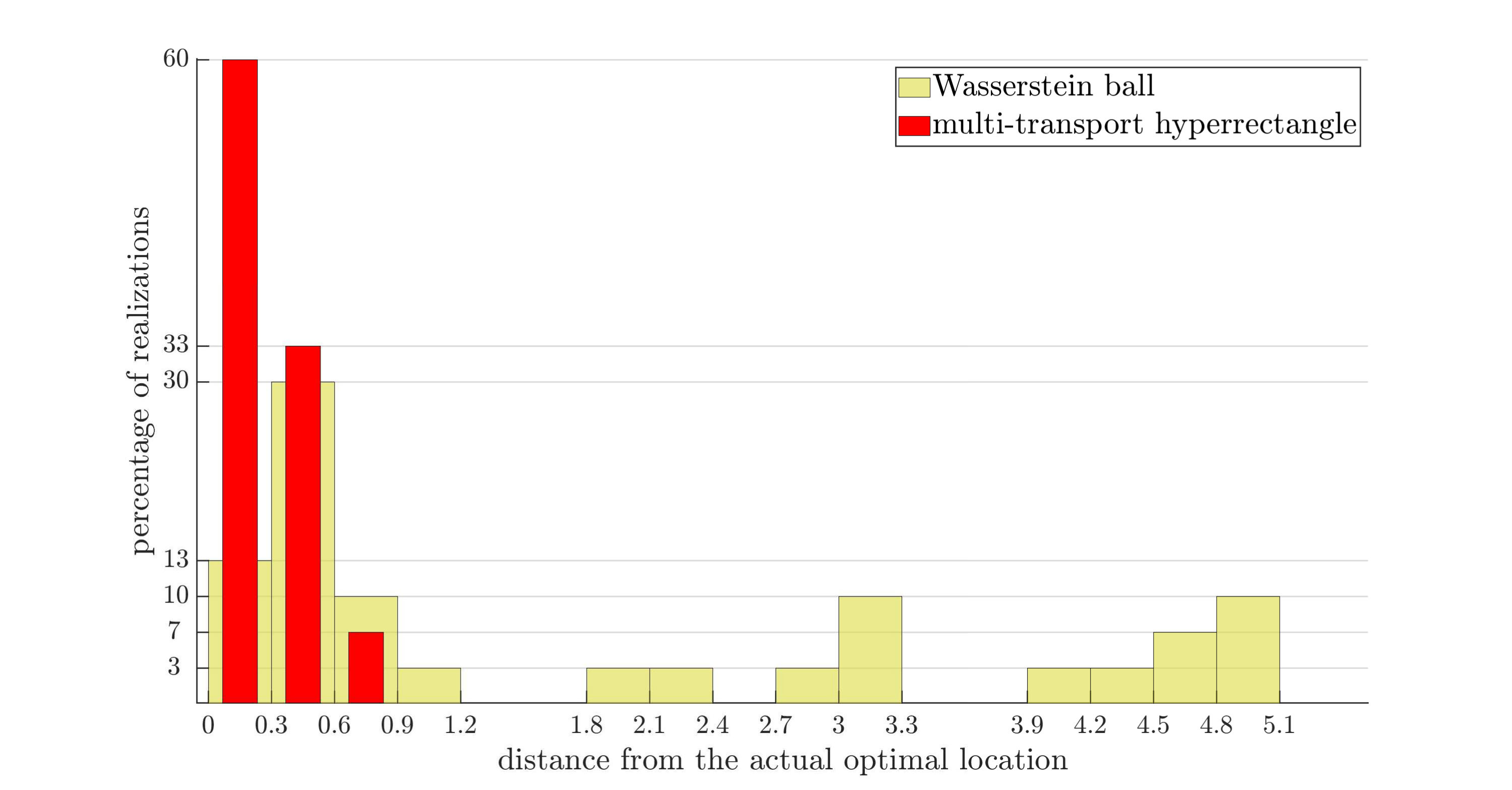}
    \caption{The histogram shows the relative frequency distribution of the distance between the true optimal location and the optimal location obtained by the DRO problem using  both ambiguity sets. The yellow bars represent the percentage of the optimal DRO locations whose discrepancies lie in the designated range when using the Wasserstein ball. The red bars represent the corresponding percentages when using the multi-transport hyperrectangle and are considerably closer to zero.}
    \label{fig:DistFromYellowPoint}
\end{figure}

For the simulations, we selected the probability distributions  $P_{\xi_1}=P_{\xi_2}:=0.1\mathcal U_{\Theta_1}+0.9\mathcal U_{\Theta_2}$ for the positions of the drones, where $\Theta_1:=\{\xi\in\Rat{2}:(0,0)\preceq \xi\preceq (2,2)\}$, $\Theta_2:=\{\xi\in\Rat{2}:(-20,-22)\preceq \xi\preceq (0,0)\}$, and  $\mathcal U$ denotes the uniform distribution on the designated set. The transport budgets for the ambiguity sets are taken as $\bm\eps=(0.01,0.01)$ and $\eps=\|\bm\eps\|$, respectively, while both sets are built using $N=50$ i.i.d. samples from $P_\xi=P_{\xi_1}\otimes P_{\xi_2}$. 

The results indicate that accounting for the product structure of the true distribution highly improves the performance of the DRO solution, even when the ambiguity sets $\mathcal{T}_2(\bm P_\xi^N,\bm\eps)$ and $\mathcal{B}_2(P_\xi^N,\eps)$ have a comparable size. Figure~\ref{fig:RechargingLocations} shows the optimal recharging locations for both ambiguity sets across 30 realizations of the simulation  while Figure~\ref{fig:DistFromYellowPoint} depicts the distribution of the distances between the true optimal location and the ones computed using the ambiguity sets across these realizations. Clearly, the multi-transport hyperrectangle exhibits superior performance compared to the Wasserstein ball. This can be certified by the fact that $60$\% of the locations obtained by solving the problem with the multi-transport hyperrectangle have a distance at most  $0.3$ from the optimal location, while this happens for only a $13$\% of the locations obtained with the Wasserstein ball. In addition, the spread of this discrepancy for the solutions obtained with the Wasserstein ball is considerably larger.

\section{Conclusion}

In this paper, we introduced two classes of structured ambiguity sets, termed Wasserstein hyperrectangles and multi-transport hyperrectangles. In data-driven scenarios where the components of the uncertainty are statistically independent, both ambiguity sets can be tuned to contain the true distribution with prescribed confidence while exhibiting considerably faster shrinkage with the number of samples compared to monolithic ambiguity balls. We established strong duality results for DRO problems over both ambiguity sets and clarified the tradeoff between the scope of the problems that can be effectively solved for each set and the potential conservativeness of the solutions to these problems. Our numerical results certify how structured ambiguity sets can capture the uncertainty in a more effective manner than monolithic ambiguity balls and improve the task of distributionally robust decision-making.  

Future work includes the development of tractable reformulations for specific classes of distributionally robust optimization and chance-constrained problems.  We also seek to address the computational complexity of data-driven problems where the product empirical distribution is supported on a prohibitive amount of points and extend our statistical analysis for structured dependencies across the components of the uncertainty.     

\appendix 

\section{Proofs from Sections~\ref{sec:hyperrectangles} and \ref{sec:statistical:guarantees}}

\subsection{Proofs from Section~\ref{sec:hyperrectangles}}  
\label{appendix:to:sec:hyperrectangles}

The following lemma is used to prove Lemma~\ref{lemma independence}.
\begin{lemma} \longthmtitle{Independent $\sigma$-algebras \cite[Theorem 2.26]{AK:13}} Let $K$ be an arbitrary set and $I_k \;, k \in K$, arbitrary mutually disjoint index sets. Define $I =\cup_{k\in K} I_k$ . If the family $\{X_i\}_{i\in I}$ is independent, then the family of $\sigma$-algebras $\{\sigma(X_j,j\in I_k)\}_{k\in K}$ is independent.
\label{lemma:independence:groups1} 
\end{lemma}
\begin{proof}[Proof of Lemma \ref{lemma independence}]
Consider the index sets $I_k=\{(k,1),\ldots,(k,N)\}$ for $k\in K:=[n]$ and let  $I:=\cup_{k\in K}I_k$. Denote by $\mathcal F_{I_k}$, $k\in K$, the $\sigma$-algebra generated by $\{\xi_k^i\}_{(k,i)\in I_k}$, where $\xi_k^i$ denotes the $k$th component of the $i$th sample $\xi^i$. Then by Assumption~\ref{assumption:independence}, the fact that $\xi^1,\ldots,\xi^N$ are i.i.d., and  Lemma~\ref{lemma:independence:groups1}, the $\sigma$-algebras $\mathcal F_{I_k}$ are independent.  Next, since  $W_p(\mu_X^N,\mu_Y^N)\le (\frac{1}{N}\sum_{i=1}^N\rho(X^i,Y^i)^p)^\frac{1}{p}$ for any discrete distributions $\mu_X^N=\frac{1}{N}\sum_{i=1}^N\delta_{X^i}$ and  $\mu_Y^N=\frac{1}{N}\sum_{i=1}^N\delta_{Y^i}$ on the Polish space $\Xi$ (cf. \cite[proof of Lemma A.2]{DB-JC-SM:21-tac}), we deduce that each mapping  
\begin{align*}
\Xi_k^{N}\ni(\xi_k^1,\ldots,\xi_k^N)\mapsto W_p(P_{\xi_k}^N,P_{\xi_k})\in\mathbb R
\end{align*}
is continuous, and hence, also measurable. Thus, $\sigma(W_p(P_{\xi_k}^N,P_{\xi_k}))\subset\mathcal F_{I_k}$ for each $k\in[n]$ and since the $\sigma$-algebras $\mathcal F_{I_k}$ are independent, the events $\{W_p(P_{\xi_k}^N,P_{\xi_k})\le\varepsilon_k\}$, $k\in[n]$ are  also independent.
\end{proof}

For the proof of Proposition~\ref{prop:hyperrectangles:geomtery} we will use the following auxiliary results, which relate the Wasserstein distance of two distributions in a product space with their transport cost discrepancy across the components of the product.  

\begin{prop}\label{prop:Wasserstein:product}
\longthmtitle{Wasserstein distance of transport distributions}
Consider the distributions $P,Q\in\mathcal P_p(\Xi)$, where  $\Xi=\Xi_1\times\cdots\times\Xi_n$ is endowed with the metric $\rho:=\big(\sum_{k=1}^n\rho_k^q\big)^{1/q}$ for some $q\ge1$ and $\rho_k$ is the metric on $\Xi_k$. Assume also that there exists a transport plan  $\pi\in\mathcal C(Q,P)$ with 
\begin{align}
\int_{\Xi\times\Xi} \rho_k(\zeta_k,\xi_k)^pd\pi(\zeta,\xi)\le\varepsilon_k^p,
\quad k\in[n], \label{transport budget constraints}
\end{align}
for certain $\eps_k>0$. Then
\begin{align}
  W_p^p(Q,P)\le n^{\max\{0,p/q-1\}} \sum_{k=1}^n \varepsilon_k^p.  \label{equation prop:Wasserstein:product}
\end{align}
\end{prop}

\begin{proof}
From the definition of the Wasserstein distance and the inequality 
\begin{align*}
\Big(\sum_{k=1}^n a_k\Big)^\gamma\le n^{\max\{0,\gamma-1\}}\sum_{k=1}^n a_k^\gamma,
\end{align*}
which holds for all $\gamma\ge 0$ and $a_k\ge 0$, we deduce that  
\begin{align*}
    W_p^p(Q,P) & \le \int_{\Xi\times\Xi} \rho(\zeta,\xi)^pd\pi(\zeta,\xi)=\int_{\Xi\times\Xi}\Big(\sum_{k=1}^n\rho_k(\zeta_k,\xi_k)^q\Big)^{p/q}d\pi(\zeta,\xi), \\
    & \le\int_{\Xi\times\Xi} n^{\max\{0,p/q-1\}} \sum_{k=1}^{n}\rho_k(\zeta_k,\xi_k)^p d\pi(\zeta,\xi) \le n^{\max\{0,p/q-1\}} \sum_{k=1}^{n}\varepsilon_k^p,
\end{align*}
where we used  \eqref{transport budget constraints} in the last inequality. This establishes \eqref{equation prop:Wasserstein:product}.
\end{proof}

\begin{prop} \label{prop:Wasserstein:product2}
Consider the product distributions $Q=Q_1\otimes\cdots\otimes Q_n$ and $P=P_1\otimes\cdots\otimes P_n$ on $\Xi=\Xi_1\times\cdots\times\Xi_n$ endowed with the metric $\rho:=\big(\sum_{k=1}^n\rho_k^p\big)^{1/p}$, where $p\ge1$ and $\rho_k$ is the metric on $\Xi_k$. Then  
\begin{align}
W_p^p(Q,P)=\sum_{k=1}^n W_p^p(Q_k,P_k).
\label{eq:prop:Wasserstein:product:equality}
\end{align}
\end{prop}

\begin{proof} By Kantorovich duality for the transport costs  $W_p^p(Q_k,P_k)$ (cf. Section~\ref{sec:prelims}), we get that 
\begin{align}
    \sum_{k=1}^n W_p^p(Q_k,P_k) & = \sum_{k=1}^n \sup_{\substack{(\psi_k,\phi_k)\in L^1(Q_k)\times L^1(P_k) \\ \phi_k(\xi_k)-\psi_k(\zeta_k) \le \rho_k(\zeta_k,\xi_k)^p}} \Big\{\int_{\Xi_k} \phi_k(\xi_k)dP_k(\xi_k)-\int_{\Xi_k} \psi_k(\zeta_k) dQ_k(\zeta_k)\Big\} \nonumber \\
        & = \sup_{\substack{(\psi_k,\phi_k)\in L^1(Q_k)\times  L^1(P_k) \\ \phi_k(\xi_k)-\psi_k(\zeta_k) \le \rho_k(\zeta_k,\xi_k)^p ,\; k\in[n]}} \sum_{k=1}^n \Big\{\int_{\Xi_k}  \phi_k(\xi_k) dP_k(\xi_k) -\int_{\Xi_k} \psi_k(\zeta_k) dQ_k(\zeta_k)\Big\} \nonumber \\ 
        & =  \sup_{\substack{(\psi,\phi)\in L^1(Q)\times L^1(P)\\
        \psi=\sum_{k=1}^n\psi_k\circ{\rm pr}_k,
        \phi=\sum_{k=1}^n\phi_k\circ{\rm pr}_k \\
        (\psi_k,\phi_k)\in L^1(Q_k)\times  L^1(P_k) \\
        \phi_k(\xi_k)-\psi_k(\zeta_k) \le \rho_k(\zeta_k,\xi_k)^p ,\; k\in[n]}}\Big\{\int_\Xi\phi(\xi)dP(\xi) -\int_\Xi\psi(\zeta) dQ(\zeta)\Big\} \nonumber \\
        & \le\sup_{\substack{(\psi,\phi)\in L^1(Q)\times  L^1(P) \\ \phi(\xi)-\psi(\zeta)\le\sum_{k=1}^n \rho_k(\zeta_k,\xi_k)^p}}\Big\{\int_\Xi\phi(\xi)dP(\xi)-\int_\Xi\psi(\zeta)dQ(\zeta)\Big\}. \label{eq: dual reformulation 1}
\end{align}
Here the second equality follows from the fact that the  constraints on $\psi_k$ and $\phi_k$, $k\in[n]$ are  decoupled, 
and the last equality from the fact that whenever $\phi_k\in L^1(P_k)$ for all $k\in[n]$ and $\phi=\sum_{k=1}^n\phi_k\circ{\rm pr}_k$, then $\phi\in L^1(P)$ (analogously for $\psi_k$, $\psi$) and  
\begin{align*}
\int_\Xi\phi(\xi)dP(\xi)=\int_\Xi\sum_{k=1}^n\phi_k\circ{\rm pr}_k(\xi)dP(\xi)=\sum_{k=1}^n\int_{\Xi_k}\phi_k(\xi_k)dP_k(\xi_k) 
\end{align*}
(and analogously for $\int_\Xi\psi(\zeta)dQ(\zeta)$). Since $\sum_{k=1}^n\rho_k(\zeta_k,\xi_k)^p=\rho(\zeta,\xi)^p$, we get from \eqref{eq: dual reformulation 1} that 
\begin{align*}
\sum_{k=1}^nW_p^p(Q_k,P_k) & \le\sup_{\substack{(\psi,\phi)\in L^1(Q)\times L^1(P)\\ \phi(\xi)-\psi(\zeta) \le  \rho(\zeta,\xi)^p}}\Big\{\int_\Xi\phi(\xi)dP(\xi)-\int_\Xi \psi(\zeta) dQ(\zeta)\Big\}= W_p^p(Q,P).
\end{align*}
Conversely, following the exact same steps as in the first part of the proof of Proposition~\ref{prop:containment} and using again the fact that $\sum_{k=1}^n\rho_k(\zeta_k,\xi_k)^p=\rho(\zeta,\xi)^p$, it follows that also $W_p^p(Q,P)\le\sum_{k=1}^nW_p^p(Q_k,P_k)$. This establishes \eqref{eq:prop:Wasserstein:product:equality} and concludes the proof.
\end{proof}

\begin{proof}[Proof of Proposition~\ref{prop:hyperrectangles:geomtery}]
From the definition of the multi-transport hyperrectangle and Proposition \ref{prop:Wasserstein:product}, we have that
\begin{align*}
W_p^p(Q,P)\le n^{\max\{0,p/q-1\}}\sum_{k=1}^n\eps_k^p
\end{align*}
for all $P\in \mathcal T_p(Q,\bm\eps)$. Therefore, $\mathcal T_p(Q,\bm\eps)\subset \mathcal B_p(Q,\eps)$. To prove the second part of the statement, assume that $Q$ is the product measure  $Q_1\otimes\cdots\otimes Q_n$ and consider probability distributions $P_k,\; k\in[n]$ such that $W_p^p(Q_k,P_k)=\eps_k^p$. Then it follows from the construction of the Wasserstein hyperrectangle \eqref{hyperrectangle} (with $\bm P_\xi^N\equiv Q$) that $P=P_1\otimes\cdots\otimes P_n\in\mathcal H_p(Q,\bm\eps)$ and we get from Proposition~\ref{prop:containment} that also $P\in\mathcal T_p(Q,\bm\eps)$. Further, since $p=q$, we obtain from  Proposition \ref{prop:Wasserstein:product2} that $       W_p^p(Q,P)=\sum_{k=1}^n\eps_k^p=\eps$, which concludes the proof.
\end{proof}

\subsection{Proofs from Section~\ref{sec:statistical:guarantees}}
\label{appendix:to:sec:statistical:guarantees}

\begin{proof}[Proof of Proposition~\ref{prop:hyperrectangle:containment}] 
For each component of the Wasserstein hyperrectangle, we consider the confidence level $1-\beta_k$ with $\beta_k$ as given in the statement. Then we get from Corollary~\ref{cor:confidence} that $\mathcal T_p(\bm P_\xi^N,\bm \eps)$ contains $P_\xi$ with confidence
\begin{align*}
\prod_{k=1}^n(1-\beta_k)\ge 1-\sum_{k=1}^n\beta_k=1-\sum_{k=1}^n\beta\frac{d_k}{d}=1-\beta. 
\end{align*}
Denoting $r_k:=d/d_k$, we get from the definition of $\widehat C$ that 
\begin{align*}
\frac{\widehat C(\beta,d)}{\widehat C(\beta_k,d_k)}
& =r_k^{1/q}\frac{C(d,p)+(\ln\beta^{-1})^{1/2p}}{C(d_k,p) +(\ln\beta_k^{-1})^{1/2p}} \\
& \ge r_k^{1/q}\frac{C(d,p)+(\ln\beta^{-1})^{1/2p}}{C(d,p) +(\ln\beta_k^{-1})^{1/2p}} = r_k^{1/q}\frac{C(d,p)+(\ln\beta^{-1})^{1/2p}}{C(d,p) +(\ln(r_k\beta^{-1}))^{1/2p}} \\ 
& \ge r_k^{1/q}\frac{C(d,p)+(\ln\beta^{-1})^{1/2p}}{C(d,p) +(\ln\beta^{-1})^{1/2p}+(\ln r_k)^{1/2p}} \ge r_k^{1/q}\frac{1}{1+(\ln r_k)^{1/2p}} 
\ge \frac{r_k^{1/q}}{1+(\ln r_k)^{1/2}}. 
\end{align*}
For these derivations, we took into account that $C(d,p)$ is increasing with respect to $d$ in the first inequality, that $(\xi+\zeta)^a\le \xi^a+\zeta^a$ for any $x,y\ge 0$ and $a\in[0,1]$ in the second inequality, and that $C(d,p)\ge 1$ and $\ln\beta^{-1}\ge 0$ in the third inequality. Taking the derivative of the function $h(\xi)= \frac{\xi^{1/q}}{1+(\ln \xi)^{1/2}}$ for $\xi\ge 1$, we get 
\begin{align*}
{\rm sign}(\dot h(\xi)) & ={\rm sign}\Big(\frac{1}{q}\xi^{1/q-1}(1+(\ln \xi)^{1/2})-\xi^{1/q}\frac{1}{2(\ln \xi)^{1/2}}\frac{1}{\xi}\Big) \\
& ={\rm sign}\Big(\frac{1}{q}(\ln \xi+(\ln \xi)^{1/2})-\frac{1}{2}\Big)={\rm sign}\Big(\zeta^2+\zeta-\frac{q}{2}\Big),
\end{align*}
where $\zeta=(\ln \xi)^{1/2}\ge 0$. It can be checked that $h(\xi)$ attains its minimum when $\xi=e^{\zeta^2}=e^{(\sqrt{2q+1}-1)^2/4}$. We thus deduce that 
\begin{align} \label{eps:star:comparison}
\frac{\widehat C(\beta,d)}{\widehat C(\beta_k,d_k)}\ge \frac{1}{c}
\end{align}
with $c$ as given in the statement.  

Now pick any $P\in\mathcal T_p(\bm P_\xi^N,\bm\eps)$. From the definition of  $\mathcal T_p(\bm P_\xi^N,\bm\eps)$, $\int_{\Xi\times\Xi} \|\xi_k-\zeta_k \|_q^pd\pi(\xi,\zeta)\le\varepsilon_k^p$ for each $k\in[n]$. Thus, we get from \eqref{eps:star:comparison} and  Proposition~\ref{prop:Wasserstein:product} with each $\rho_k$ and $\rho$ induced by the $\|\cdot\|_q$ norm that 
\begin{align*}
W_p^p(\bm P_\xi^N,P) & \le n^{\max\{0,p/q-1\}}\sum_{k=1}^n\varepsilon_k^p =n^{\max\{0,p/q-1\}}\sum_{k=1}^n\rho_\Xi^p\widehat C(\beta_k,d_k)^pN^{-p/d_k} \\
& \le n^{\max\{0,p/q-1\}}\sum_{k=1}^nc^p\rho_\Xi^p\widehat C(\beta,d)^pN^{-p/d_{\max}}.
\end{align*}
Hence, $\mathcal T_p(Q,\bm\eps)\subset\mathcal B_p(Q,\eps)$. Together with the fact that $\mathcal H_p(Q,\bm\varepsilon)\subset\mathcal T_p(Q,\bm\eps)$ by Proposition~\ref{prop:containment}, this concludes the proof.  
\end{proof}

\section{Strong duality}
\label{appendix:duality:proof}

In this section, we prove the strong duality result of Theorem~\ref{thm:strong:duality} over multi-transport hyperrectangles. We prove the result progressively by following the approach in \cite{JB-KM:19} and placing an increasing emphasis on the parts where the necessary modifications are more significant. To this end, we first prove duality when the uncertainty space $\Xi$ is compact and the transport costs $c_1,\ldots,c_n$ are continuous and satisfy Assumption~\ref{assumption:cost:functions}(ii) with $c_{k,m}\equiv c_k$. Then we extend the result to general costs, and finally, to noncompact spaces with general costs. The duality result for compact spaces also guarantees the existence of a primal optimal transport plan, which, as in \cite{JB-KM:19}, is thereafter utilized to prove duality in the most general case. We therefore state it as a separate result.

\begin{prop}
\label{prop:duality:compact:spaces}
\longthmtitle{Duality for compact spaces}
Assume that $\Xi$ is compact, $h$ is upper semicontinuous, and  the transport costs $c_k$, $k\in[n]$ satisfy Assumption~\ref{assumption:cost:functions}. Then $\mathcal I^\star=\mathcal J_\star$ and there exists a primal optimizer $\pi^\star\in\Pi(Q,\bm\epsilon)$ with $\mathcal I(\pi^\star)=\mathcal I^\star$.
\end{prop}

\subsection{Compact uncertainty space $\Xi$ and continuous costs $c_1,\ldots,c_n$ satisfying Assumption~\ref{assumption:cost:functions}(ii) with $c_{k,m}\equiv c_k$ and $\Xi_{\rm cmp}\equiv \Xi$} 

Let $X=C(\Xi\times \Xi)$ and $X^*=\mathcal M(\Xi\times \Xi)$ be the dual pair of Banach spaces of continuous functions and finite signed measures on $\Xi\times \Xi$, equipped with the supremum and total variation norms, respectively. Next, define   
\begin{subequations} \label{domains:C:and:D}
\begin{align}
C & := \Big\{g\in X:g=\varphi\circ {\rm pr_1}+\sum_{k=1}^n\lambda_kc_k,\;\textup{for some}\;\varphi\in C(\Xi)\;{\rm and}\;\lambda_k\ge 0,k\in[n]\Big\} \label{domain:C} \\
D & := \{g\in X:g\succeq h\circ {\rm pr_2}\}. \label{domain:D}
\end{align}
\end{subequations}
Namely, $C$ comprises of all functions $g\in X$ that have the form $g(\zeta,\xi)=\varphi(\zeta)+\langle\bm\lambda,\bm c(\zeta,\xi)\rangle$ for all $\xi$, $\zeta$,
where $\varphi\in C(\Xi)$ and $\bm\lambda\succeq 0$, and $D$ of all $g\in X$ with $g(\zeta,\xi)\ge h(\xi)$ for all $\zeta,\xi$. Then we have the following result.

\begin{lemma} \label{lemma:C:D:properties}
\longthmtitle{Properties of $C$ and $D$}
(i) The sets $C$ and $D$ are nonempty and convex.

\noindent (ii) For each $g\in C$ there is a unique $(\bm \lambda,\varphi)\in\Rat{n}_{\ge 0}\times C(\Xi)$ such that $g=\varphi\circ {\rm pr_1}+\sum_{k=1}^n\lambda_kc_k$.
\end{lemma}

\begin{proof}
To show (i), note that since $h$ is upper semicontinuous, $D$ is always nonempty, while convexity of the sets $C$ and $D$ follows directly from their definitions. 
To show (ii), it suffices by  the definition of $C$ to prove that if $g=\varphi\circ {\rm pr_1}+\sum_{k=1}^n\lambda_kc_k=\varphi'\circ {\rm pr_1}+\sum_{k=1}^n\lambda_k'c_k$ for some $\bm\lambda,\bm\lambda'\in\Rat{n}_{\ge 0}$ and $\varphi,\varphi'\in C(\Xi)$, then necessarily $\bm\lambda=\bm\lambda'$ and $\varphi=\varphi'$. Indeed, by Assumption~\ref{assumption:cost:functions}(ii) with $c_{k,m}\equiv c_k$ and $\Xi_{\rm cmp}\equiv \Xi$, $\text{span}\{c_1,\ldots,c_n\}\cap C_{2,{\rm const}}(\Xi\times\Xi)=\{0\}$, which implies that $(\varphi-\varphi')\circ {\rm pr_1}=0$, and $\sum_{k=1}^n(\lambda_k-\lambda_k')c_k=0$. Hence, $\varphi=\varphi'$ and since $c_1,\ldots,c_n$ are linearly independent by Assumption~\ref{assumption:cost:functions}(ii), we get from $\sum_{k=1}^n(\lambda_k-\lambda_k')c_k=0$ that also $\bm\lambda=\bm\lambda'$.
\end{proof}
  
Next, define the functionals $\Phi,\Gamma:X\xrightarrow{}\bar{\mathbb R}$ with  
\begin{subequations}
\label{functionals:Phi:and:Gamma}
\begin{align}         
\Phi(g):= & \begin{cases}
\langle\boldsymbol{\lambda},\bm\epsilon\rangle+\int_{\Xi}\varphi(\zeta)dQ(\zeta), & {\rm if}\; g\in C \\
+\infty, & {\rm otherwise},
\end{cases} \\
\Gamma(g):= & \begin{cases}
0, & {\rm if}\; g\in D \\
+\infty, & {\rm otherwise}.
\end{cases}
\end{align}
\end{subequations}
By Lemma~\ref{lemma:C:D:properties}, both functionals $\Phi$ and $\Gamma$ are well defined, convex, and have domains $C$ and $D$, respectively. To prove Proposition \ref{thm:strong:duality}, we make use of the following lemma, which determines the conjugate functionals of $\Phi$ and $\Gamma$ and their respective domains.
\begin{lemma}
\label{lemma:dual:functionals:and:domains}
\longthmtitle{Conjugates of the functionals $\Phi$, $\Gamma$ and their domains}
Consider the functionals $\Phi$, $\Gamma$ defined in \eqref{functionals:Phi:and:Gamma}. Then their conjugate functionals $\Phi^*$, $\Gamma^*$ and their respective domains $C^*$, $D^*$ are given by
\begin{subequations} 
\begin{align}
\Phi^*(\pi) & := 
\begin{cases}
0, & {\rm if}\; g\in C^* \\
+\infty, & {\rm otherwise},
\end{cases}
\label{functional:Phi:star} \\
C^* & :=\{\pi\in \mathcal M(\Xi\times \Xi):
 \langle c_k,\pi\rangle\le\epsilon_k\;\textup{for all}\;k\in[n]\;\textup{and}\;{\rm pr}_{1\#}\pi=Q\}
\label{domain:C:star} 
\end{align}
\end{subequations}
%
and
\begin{subequations} 
\begin{align} 
\Gamma^*(\pi) & : =
\begin{cases}
\int_{\Xi\times\Xi} h(\xi)d\pi(\zeta,\xi) , & {\rm if}\; g\in D^* \\
+\infty, & {\rm otherwise},
\end{cases}
\label{functional:Gamma:star}\\
D^* & : =\{\pi\in\mathcal M(\Xi\times \Xi):\langle h\circ{\rm pr}_2,\pi\rangle <+\infty\;{\rm and}\;\pi\preceq 0\}, \label{domain:D:star}
\end{align}   
\end{subequations}
where $\langle\cdot,\cdot\rangle$ denotes the duality between  $C(\Xi\times \Xi)$ and $\mathcal M(\Xi\times \Xi)$ and the order $\succeq$ is considered with the respect to the positive cone in $\mathcal M(\Xi\times \Xi)$.   
\end{lemma}

\begin{proof}
The conjugate functionals $\Phi^*$ and $\Gamma^*$ are equivalently defined as
\begin{align*}
\Phi^*(\pi):=\sup_{g\in C}\Big\{\int_{\Xi\times\Xi} g(\zeta,\xi)d\pi(\zeta,\xi)-\Phi(g) \Big\}\quad\textup{and}\quad\Gamma^*(\pi):=\sup_{g\in D}\int_{\Xi\times\Xi} g(\zeta,\xi)d\pi(\zeta,\xi)
\end{align*}
and their domains are the subsets of $\mathcal M(\Xi\times \Xi)$ for which their values are finite. To determine $\Phi^*$ and $C^*$, we get from Lemma~\ref{lemma:C:D:properties}(ii) that for every $\pi\in \mathcal M(\Xi\times \Xi)$, 
\begin{align*}
& \sup_{g\in C}\Big\{\int_{\Xi\times\Xi} g(\zeta,\xi)d\pi(\zeta,\xi)-\Phi(g) \Big\} \\
& \quad = \sup_{(\bm\lambda,\varphi)\in\mathbb R_{\ge 0}^n\times C(\Xi)}\Big\{\int_{\Xi\times\Xi}(\varphi(\zeta)+\langle \boldsymbol{\lambda},\boldsymbol{c}(\zeta,\xi)\rangle)d\pi(\zeta,\xi)-\Big(\langle\boldsymbol{\lambda},\boldsymbol\epsilon\rangle+\int_{\Xi}\varphi(\zeta)dQ(\zeta)\Big) \Big\}\\
& \quad =\sup_{(\bm\lambda,\varphi)\in\mathbb R_{\ge 0}^n\times C(\Xi)}\Big\{\sum_{k=1}^n\lambda_k\Big(\int_\Xi c_k(\zeta,\xi)d\pi(\zeta,\xi)-\epsilon_k \Big) +\int_{\Xi\times\Xi} \varphi(\zeta)d({\rm pr}_{1\#}\pi-Q)(\zeta) \Big\},\\
& \quad =\begin{cases}
0, & {\rm if} \int c_k(\zeta,\xi)d\pi(\zeta,\xi)\le\epsilon_k\;\textup{for all}\;k\in[n]\;\text{and}\;{\rm pr}_{1\#}\pi=Q \\
+\infty, & {\rm otherwise,}
\end{cases}
\end{align*}
which establishes \eqref{functional:Phi:star} and \eqref{domain:C:star}. 

To determine $\Gamma^*$ and $D^*$, we get by the exact same arguments as in the respective part of the proof of \cite[Proposition 1]{JB-KM:19} that 
\begin{align*}
\sup_{g\in D}\int_{\Xi\times\Xi} g(\zeta,\xi)d\pi(\zeta,\xi)
=\begin{cases}
\int_{\Xi\times\Xi} h(\xi)d\pi(\zeta,\xi), & {\rm if}\;\pi\in\mathcal M(\Xi\times \Xi)\; \textup{is non-positive} \\
+\infty, & {\rm otherwise},
\end{cases}
\end{align*}
which implies \eqref{functional:Gamma:star} and \eqref{domain:D:star}.
\end{proof}

We also make use of the Fenchel-Rockafellar duality theorem, which we state below. This form of the theorem is more convenient to verify in our setting compared to the more general form invoked in \cite{JB-KM:19}, where one needs to verify conditions about the relative interior of the involved functionals that are harder to show in our case.

\begin{thm}
\longthmtitle{Fenchel-Rockafellar duality \cite[Theorem 1.12]{HB:10}} 
\label{thm:Fenchel:Rockafellar:duality}
Let $X$ be a normed vector space, $X^*$ its topological dual, and $\Phi,\Gamma:X\to \mathbb R\cup\{+\infty\}$ two convex functionals with domains $C$ and $D$, repectively. Assume further that there is some $x_0\in C\cap D$ so that $\Gamma$ is continuous at $x_0$. Then 
\begin{align*}
\inf_{x\in X}\{\Phi(x)+\Gamma(x)\}=\max_{x^*\in X^*}\{-\Phi^*(x^*)-\Gamma^*(-x^*)\},
\end{align*}
where $\Phi^*$ and $\Gamma^*$ are the conjugates of $\Phi$ and $\Gamma$.  
\end{thm}

We now give the proof of Proposition~\ref{prop:duality:compact:spaces} for suitable continuous costs. 

\begin{proof}[Proof of Proposition~\ref{prop:duality:compact:spaces}] \longthmtitle{For continuous costs $c_1,\ldots,c_n$ that satisfy Assumption~\ref{assumption:cost:functions}(ii)}
From the definition of the functionals $\Phi$ and $\Gamma$ in \eqref{functionals:Phi:and:Gamma} and their respective domains $C$ and $D$ in \eqref{domains:C:and:D}, we have that  
\begin{align*}
\inf_{g\in X}\{\Phi(g)+\Gamma(g)\}=\inf_{g\in C\cap D}\{\Phi(g)+\Gamma(g)\}=\inf\{\mathcal J(\bm\lambda,\varphi):(\bm\lambda,\varphi)\in\Lambda\;{\rm and}\;\varphi\in C(\Xi)\},
\end{align*}
with $\mathcal J$ and $\Lambda$ as given in \eqref{Lambda:and:J}. By Lemma \ref{lemma:dual:functionals:and:domains} and \eqref{transport:plan:set}, their conjugate functionals $\Phi^*$, $\Gamma^*$ and their domains $C^*$, $D^*$ satisfy 
\begin{align*}
-\Phi^*(\pi)-\Gamma^*(-\pi)=\int_{\Xi\times\Xi} h(\xi)d\pi(\zeta,\xi) 
\end{align*}
and $C^*\cap -D^*=\Pi(Q,\bm\epsilon)$. Thus, we get from \eqref{inner:maximization:reformulation} that
\begin{align*}
\sup_{\pi\in X^*}\{-\Phi^*(\pi)-\Gamma^*(-\pi)\}=\sup_{\pi\in C^*\cap -D^*}\{-\Phi^*(\pi)-\Gamma^*(-\pi)\}=\mathcal I^\star.
\end{align*}
Next, by exploiting upper semicontinuity of $h$, there exists an element $g_0\in C\cap D$ where $\Gamma$ is continuous. For example, we may take $g_0(\zeta,\xi):=\sup_{\xi\in \Xi}h(\xi)+1$, which implies that $g\in D$ for all $g$ in a neighborhood of $g_0$ in $X$, and thus, that $\Gamma(g)=0$, which establishes continuity at~$g_0$. Consequently, we deduce from Theorem~\ref{thm:Fenchel:Rockafellar:duality} that 
\begin{align*}
\inf_{g\in C\cap D}\{\Phi(g)+\Gamma(g)\}=\max_{\pi\in X^*}\{-\Phi^*(\pi)-\Gamma^*(-\pi)\}, 
\end{align*}
where the max on the right is attained for some $\pi^*\in \mathcal M(\Xi\times\Xi)$. We claim that $\pi^*\in C^*\cap -D^*$. Otherwise, if $\pi^*\in \mathcal M(\Xi\times\Xi)\setminus(C^*\cap -D^*)$, we would have that  
$\mathcal I^\star=-\Phi^*(\pi^*)-\Gamma^*(-\pi^*)=-\infty$. But this is a contradiction because $h$ is integrable with respect to $Q$ and $\Pi(Q,\bm\epsilon)$ is nonempty, since $c_k(\zeta,\zeta)\equiv 0$ for all $k$. 
We therefore get that 
\begin{align*}
\inf\{\mathcal J(\bm\lambda,\varphi):(\bm\lambda,\varphi)\in\Lambda\;{\rm and}\;\varphi\in C(\Xi)\}=\max_{\pi\in\Pi(Q,\bm\epsilon)}\mathcal I(\pi)=\mathcal I^\star 
\end{align*}
and since $C(\Xi)\subset\mathfrak{m}_{\mathcal U}(\Xi;\mathbb R\cup\{+\infty\})$, it follows from \eqref{J:star} that
\begin{align*}
\mathcal J_\star\le\inf\{\mathcal J(\bm\lambda,\varphi):(\bm\lambda,\varphi)\in\Lambda\;{\rm and}\;\varphi\in C(\Xi)\}=\mathcal I^\star. 
\end{align*}
Combined with \eqref{weak:duality}, this concludes the proof. 
\end{proof}

\subsection{Compact uncertainty space $\Xi$ and general costs $c_1,\ldots,c_n$ satisfying Assumption~\ref{assumption:cost:functions}(ii)} 

In this section, we clarify how the machinery of the previous section can be used to establish the duality result of Proposition~\ref{prop:duality:compact:spaces} in the general case.   

\begin{proof}[Proof of Proposition~\ref{prop:duality:compact:spaces} (sketch)] 
The proof consists of minor modifications of the arguments in  \cite[proof of Proposition 2]{JB-KM:19}. Notice first that due to Assumption~\ref{assumption:cost:functions}, we automatically get that for each $m$, $c_{k,m}$, $k\in[n]$ are linearly independent in $C(\Xi\times\Xi)$ and $\text{span}\{c_{1,m},\ldots,c_{n,m}\}\cap C_{2,{\rm const}}(\Xi\times\Xi)=\{0\}$. Thus it follows from the validity of the proposition for continuous costs that there exists a sequence $\{\pi^\star_m\}$ of primal optimizers for the problems
\begin{align*}
\mathcal I^\star_m:=\sup_{\pi\in\mathcal\mathcal \Pi(Q,\bm\epsilon;c_{1,m},\ldots,c_{n,m})}\mathcal I(\pi)=\mathcal I(\pi_m^\star),
\end{align*}
whose corresponding ambiguity sets are defined through the costs $c_{k,m}$, $k\in[n]$. By the duality result  of the same proposition, we have that $\mathcal I^\star_m=\mathcal J_{m,\star}$, where        
\begin{align*}
\mathcal J_{m,\star}:=\inf_{(\bm\lambda,\phi)\in\Lambda(h;c_{1,m},\ldots,c_{n,m})}\mathcal J(\bm\lambda,\phi)
\end{align*}
are the optimal values of the corresponding dual problems. By tightness of the sequence $\{\pi_m^\star\}$, a subsequence $\{\pi_{m_\ell}^\star\}$ converges weakly to a probability measure $\pi^\star\in\mathcal P(\Xi\time\Xi)$. Then the remaining proof hinges on showing that (i) $\pi^\star\in\Pi(Q,\bm\epsilon;c_1,\ldots,c_n)$ and (ii) that $\mathcal I(\pi^\star)\ge\mathcal J_\star$, which by weak duality establishes that $\mathcal I^\star=\mathcal J_\star$ and that $\pi^\star$ is a primal optimizer. The establishment of (i) is based on the exact same arguments as the ones in \cite[proof of Proposition 2]{JB-KM:19} to verify that $\int_{\Xi\times \Xi}c_k(\zeta,\xi)d\pi^\star(\zeta,\xi)\le \epsilon_k$ for all $k\in[n]$ and that ${\rm pr}_{1\#}\pi^\star=Q$. The establishment of (ii) also follows the same arguments as the ones in  \cite[proof of Proposition 2]{JB-KM:19}. It exploits that $\Lambda(h;c_{1,m},\ldots,c_{n,m})\subset\Lambda(h;c_1,\ldots,c_n)$, which holds by Assumption~\ref{assumption:cost:functions}, to get that $\mathcal J_{m,\star}\ge \mathcal J_\star$ and show that $\mathcal I(\pi^\star)\ge \limsup_\ell\mathcal J_{m_\ell,\star}\ge \mathcal J_\star$. 
\end{proof}

\subsection{Duality for non-compact spaces and general costs $c_1,\ldots,c_n$ satisfying Assumption~\ref{assumption:cost:functions}(ii)} 

Here we sketch how the results of the previous sections can be used to establish Theorem~\ref{thm:strong:duality}. To this end, denote for each distribution $\pi\in\P(\Xi \times \Xi)$  
\begin{align*}
\Xi_\pi:={\rm supp}({\rm pr}_{1\#}\pi)\cup{\rm supp}({\rm pr}_{2\#}\pi)  
\end{align*}
and for each closed set $K\subset\Xi$
\begin{align*}
\Lambda(K\times K):=\big\{(\bm \lambda,\varphi) & : \bm \lambda\succeq 0,\varphi\in \mathfrak m_{\mathcal U}(K;\Rat{}\cup\{+\infty\}), \\
& \;\;{\rm and}\;\varphi\circ{\rm pr}_1(\zeta,\xi)\ge h\circ{\rm pr}_2(\zeta,\xi)-\sum_{k=1}^n\lambda_k c_k(\zeta,\xi)\;\textup{for all}\;\zeta,\xi\in K\Big\}.
\end{align*}
Let also
\begin{align*}
\Pi_{{\rm fin},\bm c,h}(Q):=\Big\{\pi\in\Pi_{{\rm fin},\bm c}(Q):\int_{\Xi\time\Xi}h(\xi)d\pi(\zeta,\xi)\in\Rat{}\Big\},
\end{align*}
with $\Pi_{{\rm fin},\bm c}(Q)$ as defined in Section~\ref{subsec:duality:multi:transport}. For each transport plan $\pi\in\Pi_{\rm fin,\bm c,h}(Q)$, the set $\Xi_\pi\times\Xi_\pi$, which contains the support of $\pi$, can be exhausted through a sequence of compact sets over which the integrals of the objective function $h$ and the costs $c_k$ are uniformly bounded. This makes it possible to use the result of the previous section and obtain bounds for the values of the dual problem over non-compact subsets of the space $\Xi$. In particular, we have the following auxiliary result, which will be used for the proof of the main theorem. 

\begin{prop}
\longthmtitle{Dual value bounds}
\label{prop:duality:inequality}
Let $h$ and the cost functions $c_1,\ldots,c_n$ satisfy Assumptions \ref{assumption:h} and \ref{assumption:cost:functions}, respectively. Then for any $\pi\in\Pi_{{\rm fin},\bm c,h}(Q)$, it holds that 
\begin{align*}
\inf_{(\bm \lambda,\varphi)\in\Lambda(\Xi_\pi\times\Xi_\pi)}\mathcal J(\bm \lambda,\varphi)\le \mathcal I^\star.
\end{align*}
\end{prop}

\begin{proof}[Proof (sketch)] 
The proof consists again of minor modifications of the proof of \cite[Proposition 3]{JB-KM:19}. The first step is to pick an increasing sequence of compact subsets $\Xi_m\times\Xi_m$ of $\Xi_\pi\times\Xi_\pi$ with  
\begin{align*}
p_m:=\pi(\Xi_m\times\Xi_m) & \ge 1-\frac{1}{m} \\
\int_{(\Xi_m\times \Xi_m)^c} c_k(\zeta,\xi)d\pi(\zeta,\xi) & \le\frac{\epsilon_k}{m}\; \textup{for all}\;k\in[m] \\ 
\int_{(\Xi_m\times \Xi_m)^c} |h(\xi)|d\pi(\zeta,\xi) & \le\frac{1}{m}, 
\end{align*}
where $(\Xi_m\times \Xi_m)^c:=\Xi_\pi\times \Xi_\pi\setminus(\Xi_m\times \Xi_m)$. 
For each $m$, we denote by $\pi_m$ the normalized restriction of $\pi$ to $\Xi_m\times\Xi_m$, $Q_m$ its corresponding first marginal, and $\bm\epsilon^m:=(\epsilon_1^m,\ldots,\epsilon_n^m)$, with $\epsilon_k^m:=\epsilon_k(1-\frac{1}{m})$. From   Proposition~\ref{prop:duality:compact:spaces} applied to the restriction of the DRO problem over each space $\Xi_m$ with $\Pi(Q_m,\bm\epsilon^m)$ as the associated ambiguity set, there is a zero duality gap between the values of the corresponding primal and dual problems, and there exists a primal feasible transport plan $\pi_m^\star$. Namely, 
\begin{align*}
\int_{\Xi\times\Xi} h(\xi)d\pi_m^\star(\zeta,\xi)=\mathcal I_m^\star=\mathcal J_{m,\star}.
\end{align*}
Gluing the $p_m$-weighted version of each optimal transport plan  $\pi_m^\star$ with the restriction of $\pi$ on the corresponding residual set $(\Xi_m\times\Xi_m)^c$, one can deduce by the exact same arguments as in \cite[proof of Proposition 3]{JB-KM:19} that 
\begin{align} \label{Istar:bound}
\mathcal I^\star\ge p_m\mathcal J_{m,\star}-\frac{1}{m}. 
\end{align}
By selecting $\varepsilon$-optimal vectors $\bm\lambda^m=(\lambda_1^m,\ldots,\lambda_n^m)$ of dual parameters for each dual optimal value $\mathcal J_{m,\star}$, it follows in analogy to \cite[proof of Proposition 3]{JB-KM:19} that 
\begin{align*}
\langle\bm\lambda^m,\bm\epsilon^m\rangle +\int_{\Xi_m}\sup_{\xi\in\Xi_m}\Big\{h(\xi)-\sum_{k=1}^n\lambda_k^mc_k(\zeta,\xi)\Big\}dQ_m(\zeta)\le \mathcal J_{m,\star}+\eps,
\end{align*}
which together with \eqref{Istar:bound} implies that 
\begin{align*}
\limsup_{m\to\infty}\Big\{p_m\langle\bm\lambda^m,\bm\epsilon^m\rangle +\int_{\Xi_\pi\times\Xi_\pi }\sup_{\xi\in\Xi_m}\Big\{h(\xi)-\sum_{k=1}^n\lambda_k^mc_k(\zeta,\xi)\Big\}\mathds 1_{\Xi_m\times\Xi_m}(\zeta,\xi') d\pi(\zeta,\xi')\Big\}\le \mathcal I^\star+\varepsilon.
\end{align*}
One can then show as in \cite[proof of Proposition 3]{JB-KM:19} that the sequences $\{\lambda_k^m\}_{m\in\mathbb N}$, $k\in[n]$ are bounded. Thus, there exists a subsequence $\{\bm\lambda^{m_\ell}\}_{\ell\in\mathbb N}$ converging to some $\bm\lambda^\star\succeq 0$ and it can be checked along the lines of \cite[proof of Lemma B.7]{JB-KM:19} that 
\begin{align*}
\liminf_{\ell\to\infty}\sup_{\xi\in\Xi_{m_\ell}}\Big\{h(\xi)-\sum_{k=1}^n\lambda_k^{m_\ell}c_k(\zeta,\xi)\Big\}
\ge \sup_{\xi\in\cup_{\ell=1}^\infty\Xi_{m_\ell}}\Big\{h(\xi)-\sum_{k=1}^n\lambda_k^{m_\ell}c_k(\zeta,\xi)\Big\}
\end{align*}
for all $\zeta\in\Xi_\pi$. Using the  same arguments as in \cite[proof of Proposition 3]{JB-KM:19}, this implies that  
\begin{align*}
\mathcal I^\star+\varepsilon & \ge\liminf_{m\to\infty}\Big\{p_m\langle\bm\lambda^m,\bm\epsilon^m\rangle +\int_{\Xi_\pi\times\Xi_\pi }\sup_{\xi\in\Xi_m}\Big\{h(\xi)-\sum_{k=1}^n\lambda_k^mc_k(\zeta,\xi)\Big\}\mathds 1_{\Xi_m\times\Xi_m}(\zeta,\xi') d\pi(\zeta,\xi')\Big\} \\
& \ge \langle\bm\lambda^\star,\bm\epsilon\rangle +\int_{\Xi_\pi\times\Xi_\pi }\sup_{\xi\in\cup_{\ell=1}^\infty\Xi_{m_\ell}}\Big\{h(\xi)-\sum_{k=1}^n\lambda_k^\star c_k(\zeta,\xi)\Big\}d\pi(\zeta,\xi') \\
& =\langle\bm\lambda^\star,\bm\epsilon\rangle +\int_{\Xi_\pi}\sup_{\xi\in \Xi_\pi}\Big\{h(\xi)-\sum_{k=1}^n\lambda_k^\star c_k(\zeta,\xi)\Big\}dQ(\zeta).
\end{align*}
Since $\varepsilon$ is arbitrary, selecting the pair $(\bm\lambda^\star,\varphi)\in\Lambda(\Xi_\pi\times\Xi_\pi)$ with $\varphi(\zeta):=\sup_{\xi\in \Xi_\pi}\big\{h(\xi)-\sum_{k=1}^n\lambda_k^\star c_k(\zeta,\xi)\big\}$ establishes the result. 
\end{proof}

We need one last result whose proof we omit as it is identical to that of \cite[Proposition 4]{JB-KM:19}.

\begin{prop}
\label{prop:integration:majorization:interchange}
\longthmtitle{Integration/majorization interchange}
If the objective function $h$ and the cost functions $c_1,\ldots,c_n$ satisfy Assumptions \ref{assumption:h} and \ref{assumption:cost:functions}, respectively, then
\begin{align*}
    \sup_{\pi\in\Pi_{{\rm fin},\bm c,h}(Q)} \int_{\Xi\times\Xi}(h(\xi)-\langle \bm\lambda,\bm c(\zeta,\xi) \rangle)d\pi(\zeta,\xi)=\int_{\Xi}\sup_{\xi\in\Xi}\{h(\xi)-\langle \bm\lambda,\bm c(\zeta,\xi) \rangle \}dQ(\zeta).
\end{align*}
\label{prop:measurable:selection}
\end{prop}

Now, we can proceed to sketch the proof of strong duality for general Polish spaces. 

\begin{proof}[Proof of Theorem~\ref{thm:strong:duality} (sketch)]
The proof relies on showing that $\mathcal I^\star\ge \mathcal J_\star$ and follows the steps of \cite[proof of Theroem 1]{JB-KM:19}. When $\mathcal I^\star=+\infty$, then the result follows from the fact that  $\mathcal I^\star\le \mathcal J_\star$. When $\mathcal I^\star<+\infty$, Proposition~\ref{prop:duality:inequality} implies that for each $\pi\in\Pi_{{\rm fin},\bm c,h}(Q)$
\begin{align*}
\mathcal I^\star\ge \inf_{(\bm\lambda,\varphi)\in\Lambda(\Xi_\pi\times\Xi_\pi)}
\Big\{\langle\bm\lambda,\bm\epsilon\rangle+\int_{\Xi_\pi}\varphi(\zeta)dQ(\zeta)\Big\}\ge
\inf_{\bm\lambda\succeq 0}\Big\{\langle\bm\lambda,\bm\epsilon\rangle+\int_{\Xi}\sup_{\xi\in\Xi_\pi}\{h(\xi)-\langle\bm\lambda,\bm c(\zeta,\xi)\rangle\}dQ(\zeta)\Big\}.
\end{align*}
Next, denote 
\begin{align*}
T(\bm\lambda,\pi):=\langle\bm\lambda,\bm\epsilon\rangle+\int_{\Xi}\sup_{\xi\in\Xi_\pi}\{h(\xi)-\langle\bm\lambda,\bm c(\zeta,\xi)\rangle\}dQ(\zeta)
\end{align*}
and $\lambda_{\max}:=\max_{k=1,\ldots,n}\frac{\mathcal I^\star-\int_\Xi h(\zeta)dQ(\zeta)}{\epsilon_k}$. Then it follows by the same arguments as in \cite[proof of Theroem 1(a)]{JB-KM:19} that 
\begin{align}
\mathcal I^\star\ge\inf_{\bm\lambda\in[0,\lambda_{\max}]^n}T(\bm\lambda,\pi) \label{Istar:vs:T}
\end{align}
and that $T(\bm\lambda,\pi)$ is lower semicontinuous and convex with respect to $\bm\lambda$ and concave with respect to $\pi$. Thus, it follows from  Fan's minimax theorem \cite[Theorem 2]{KF:53} that 
\begin{align*}
\sup_{\pi\in\Pi_{{\rm fin},\bm c,h}(Q)}
\inf_{\bm\lambda\in[0,\lambda_{\max}]^n}T(\bm\lambda,\pi)=\inf_{\bm\lambda\in[0,\lambda_{\max}]^n}\sup_{\pi\in\Pi_{{\rm fin},\bm c,h}(Q)}T(\bm\lambda,\pi),
\end{align*}
and we get from \eqref{Istar:vs:T} that 
\begin{align*}
\mathcal I^\star & \ge\inf_{\bm\lambda\in[0,\lambda_{\max}]^n}\Big\{\langle\bm\lambda,\bm\epsilon\rangle+\sup_{\pi\in\Pi_{{\rm fin},\bm c,h}(Q)}\int_{\Xi}\sup_{\xi\in\Xi_\pi}\{h(\xi)-\langle\bm\lambda,\bm c(\zeta,\xi)\rangle\}dQ(\zeta)\Big\} \\
& \ge\inf_{\bm\lambda\in[0,\lambda_{\max}]^n}\Big\{\langle\bm\lambda,\bm\epsilon\rangle+\sup_{\pi\in\Pi_{{\rm fin},\bm c,h}(Q)}\int_{\Xi}(h(\xi)-\langle\bm\lambda,\bm c(\zeta,\xi)\rangle)dQ(\zeta)\Big\} \\
& =\inf_{\bm\lambda\in[0,\lambda_{\max}]^n}\Big\{\langle\bm\lambda,\bm\epsilon\rangle+\int_{\Xi}\sup_{\xi\in\Xi}\{h(\xi)-\langle\bm\lambda,\bm c(\zeta,\xi)\rangle\}dQ(\zeta)\Big\},
\end{align*}
where the last equality holds due to  Proposition~\ref{prop:integration:majorization:interchange}. Since $(\bm\lambda,\varphi_{\bm\lambda})\in\Lambda$ for all $\bm\lambda\in[0,\lambda_{\max}]^n$, where $\varphi_{\bm\lambda}(\zeta):=\sup_{\xi\in\Xi}\{h(\xi)-\langle\bm\lambda,\bm c(\zeta,\xi)\rangle\}$, it follows from \eqref{J:star} that strong duality holds. 

To show that a dual optimizer of the form $(\bm\lambda,\varphi_{\bm\lambda})$ exists, let  $g(\bm\lambda):=\langle\bm\lambda,\bm\epsilon\rangle+\int_{\Xi}\varphi_{\bm\lambda}(\zeta)dQ(\zeta)$.
Then as in \cite[proof of Theroem 1(b)]{JB-KM:19}, it follows that $g$ is lower semicontinous and that $g(\bm\lambda)\ge\langle\bm\lambda,\bm\epsilon\rangle+\int_{\Xi}h(\zeta)dQ(\zeta)$, which implies that $g$ is radially unbounded since $\lim_{\|\bm\lambda\|\to+\infty}g(\bm\lambda)=+\infty$ for $\bm\lambda\succeq 0$. Hence, the level sets of $g$ are compact and its infimum is always attained. Since $\mathcal J(\bm\lambda,\varphi)\ge\mathcal J(\bm\lambda,\varphi_{\bm\lambda})$ for all $(\bm\lambda,\varphi)\in\Lambda$, the infimum of the dual problem is also attained by a pair $(\bm\lambda^\star,\varphi_{\bm\lambda^\star})\in\Lambda$ and the proof is complete. 
\end{proof}

\section{DRO reformulations of the simulation example} 

Here we derive the dual reformulations of the DRO problem \eqref{eq:simEx} in the simulation example and provide conditions under which we can remove redundant constraints in  \eqref{eq:opt:problem:ex:2}. 

\subsection{Tractable reformulations of \eqref{eq:simEx}} \label{appendix:simEx:derivations}

From Corollary \ref{cor:dual}, when $\mathcal P^N\equiv\mathcal T_2(\bm P_\xi^N,\bm\eps)$, the dual of \eqref{eq:simEx} is 
\begin{align}
    \inf_{x\in\mathcal X,\;\bm\lambda\succeq0} \langle \bm\lambda, \bm\epsilon \rangle+\frac{1}{N^2}\sum_{\bm i\in[N]^2}\sup_{\xi\in\Rat{4}}\Big\{ \mathds{1}_{\Theta\times\Theta}(\xi)\sum_{k=1}^2 \|x-\xi_k\|_2^2-\sum_{k=1}^2 \lambda_k \|\xi_{k}^{i_k}-\xi_k\|_2^2\Big\}. \label{eq:simEx:dev1}
\end{align}
Taking into account that for any pair of functions $f,g:\Xi\to \Rat{}$ with $g\le0$, $f\ge 0$, and $g(\xi^\star)=0$ for some $\xi^\star\in\Xi$, it holds that 
\begin{align*}
\sup_{\xi\in\Xi}\{\mathds{1}_K(\xi)f(\xi)+g(\xi)\}=\max\Big\{0,\sup_{\xi\in K}\{f(\xi)+g(\xi)\}\Big\} 
\end{align*}
for any $K\subset\Xi$\footnote{Indeed, let $A:=\sup_{\xi\in\Xi}\{\mathds{1}_K(\xi)f(\xi)+g(\xi)\}$, $B:=\max\{0,\sup_{\xi\in K}\{f(\xi)+g(\xi)\}\}$ and $C:=\sup_{\xi\in K}\{f(\xi)+g(\xi)\}$. If $C\ge 0$, then $B=C$ and it follows that also $A=C$, because $g\le 0$ and so the sup to get $A$ can be attained over a sequence in $K$. If $C<0$, then $B=0$ and $g<0$ on $K$ because $f\ge 0$. Thus, necessarily $g(\xi^\star)=0$ for some $\xi^\star\in\Xi\setminus K$ and we get again that $A=C=0$.}, \eqref{eq:simEx:dev1} can be written as
\begin{align}
\inf_{x\in\mathcal X,\;\bm\lambda\succeq0} \langle \bm\lambda, \bm\epsilon \rangle+\frac{1}{N^2}\sum_{\bm i\in[N]^2}\max\Big\{0,\sup_{\xi\in\Theta \times \Theta}\Big\{ \sum_{k=1}^2 \|x-\xi_k\|_2^2- \lambda_k \|\xi_{k}^{i_k}-\xi_k\|_2^2\Big\}\Big\}. \label{example:initial:DRO:reformulation}
\end{align}
Introducing epigraphical variables, and taking into account that $\Theta$ is unbounded, which implies that the sup in \eqref{example:initial:DRO:reformulation} is below $+\infty$ only when $\bm\lambda\succ\bm1$, the DRO problem becomes 
\begin{align}
\left\{\begin{aligned}
\inf_{x\in\mathcal{X},\bm\lambda\succ\bm1, \bm s\succeq 0} &\langle \bm\lambda, \bm\epsilon \rangle + \frac{1}{N^2} \sum_{\bm i\in[N]^2} s_{\bm i} \\
 \text{s.t} &\sup_{\xi\in\Theta\times \Theta}\Big\{ \sum_{k=1}^2 \|x-\xi_k\|_2^2- \lambda_k \|\xi_{k}^{i_k}-\xi_k\|_2^2\Big\} \le s_{\bm i}, \quad \bm i\in[N]^2.
 \end{aligned}\right. \label{eq:opt:problem:ex:1}
\end{align}
Now the left-hand side of each constraint is written as 
\begin{align*}
\sum_{k=1}^2 (\| x\|^2-\lambda_k\|\xi_k^{i_k}\|^2)+\sum_{k=1}^2\sup_{\xi_k\in\Theta}\big\{ \xi_k^\top (1-\lambda_k) I_2 \xi_k+2(\lambda_k\xi_k^{i_k}- x)^\top \xi_k \big\},
\end{align*}
which, since $\Theta$ is a polytope, includes two linearly constrained quadratic problems (QPs), namely, a special case of quadratically constrained QPs (QCQPs). Then we get from strong duality of QCQPs (cf. \cite[Page 227]{SPB-LV:04}) that
\begin{align*}
\sup_{\xi_k\in\Theta}\big\{ \xi_k^\top (1-\lambda_k) I_2 \xi_k &  +2(\lambda_k\xi_k^i- x)^\top \xi_k \big\} =
\inf_{\nu_k^i\succeq 0} (r_k^i-\nu_k^i)^\top \frac{1}{4(\lambda_k-1)}I_2 (r_k^i- \nu_k^i),
\end{align*}
where $r_k^i:=2(\lambda_k \xi_k^i- x)$. Thus,  each constraint in \eqref{eq:opt:problem:ex:1} is equivalent to 
\begin{align*}
\inf_{(\nu_1^{i_1},\nu_2^{i_2})\succeq0}\sum_{k=1}^2 \frac{\|r_k^{i_k}-\nu_k^{i_k}\|^2}{4(\lambda_k-1)} \le s_{\bm i}-\sum_{k=1}^2 (\| x\|^2-\lambda_k\|\xi_k^{i_k}\|^2),
\end{align*}
and taking further  into account that 
\begin{align*}
\inf_{(\nu_1^{i_1},\nu_2^{i_2})\succeq0}\sum_{k=1}^2 \frac{\|r_k^{i_k}-\nu_k^{i_k}\|^2}{4(\lambda_k-1)} =
\min_{(\nu_1^{i_1},\nu_2^{i_2})\succeq0}\sum_{k=1}^2 \frac{\|r_k^{i_k}-\nu_k^{i_k}\|^2}{4(\lambda_k-1)}=\sum_{k=1}^2 \min_{\nu_k^{i_k}\succeq0}\frac{\|r_k^{i_k}-\nu_k^{i_k}\|^2}{4(\lambda_k-1)},  
\end{align*}
the DRO problem~\eqref{eq:opt:problem:ex:1} can be cast in the form~\eqref{eq:opt:problem:ex:2}. 
Analogously, when $\mathcal P^N\equiv\mathcal B_2(P_\xi^N,\eps)$, which is essentially a single-cost multi-transport hyperrectangle, the dual of \eqref{eq:simEx} is the convex program~\eqref{eq:opt:problem:ex:3}.

\subsection{Complexity reduction of \eqref{eq:opt:problem:ex:2}} \label{appendix:NegligibleSamples:derivations}

Here we provide conditions under which certain constraints in the reformulation  \eqref{eq:opt:problem:ex:2} become redundant and can be removed, reducing the complexity of the optimization problem. These hinge on the observation that if  
\begin{align*}
    \sup_{\xi\in\Theta\times\Theta}\Big\{ \sum_{k=1}^2  \|x-\xi_k\|_2^2- \lambda_k \|\xi_{k}^{i_k}-\xi_k\|_2^2\Big\}
    \le 0
\end{align*}
for some index $(i_1,i_2)\in[N]^2$, then the corresponding constraint in \eqref{eq:opt:problem:ex:1} is always satisfied, and its epigraphical variable $s_{\bm i}$ can be omitted by setting it to 0. Since $\bm\lambda\succ\bm1$, to remove a constraint, it is sufficient to establish that $\|x-\xi_1\|_2^2\le\|\xi_{1}^{i_1}-\xi_1\|^2$ and $\|x-\xi_2\|_2^2\le\|\xi_{2}^{i_2}-\xi_2\|^2$ for all $x\in\mathcal X$ and $\xi\in\Theta$.

To this end, note that       
\begin{align}
    \sup_{\xi\in\Theta\times\Theta}\Big\{ \sum_{k=1}^2 \|x-\xi_k\|_2^2- \|\xi_{k}^{i_k}-\xi_k\|_2^2\Big\}= \|\bm x\|^2-\|\xi^{\bm i}\|^2+2\sup_{\xi\in\Theta\times\Theta} \{(\xi^{\bm i}-\bm x)^\top \xi \}, \label{eq:ex:Linprog}
\end{align}
where $\xi^{\bm i}:=(\xi_1^{i_1},\xi_2^{i_2})$ and $\bm x:=(x,x)$. The linear maximization problem on the right-hand side of \eqref{eq:ex:Linprog} can be written as 
\begin{align*}
    \sup_{\xi\in\Rat{4}}  & \; (\xi^{\bm i}-\bm x)^\top \xi \\ 
    \text{s.t.} &\; \xi\succeq 0
\end{align*}
and attains its maximum at $\xi=0$ when $\xi^{\bm i}\preceq 0$. Indeed, since $\bm x\succeq 0$, the KKT conditions hold at $\xi=0$, as all constraints are active and there exists $\gamma\in\Rat{4}$ with $\gamma\succeq 0$ and 
\begin{align*}
\xi^{\bm i}-\bm x+\sum_{i=1}^4\gamma_ie_i=0,
\end{align*}
where the $e_i$'s are the standard unit vectors in $\Rat{4}$. We therefore get from \eqref{eq:ex:Linprog} that whenever $\xi^{\bm i}\preceq 0$ and $\|\xi^{\bm i}\|\le\max_{x\in\mathcal X}\|\bm x\|$,  
\begin{align*}
\sup_{\xi\in\Theta\times\Theta}\Big\{ \sum_{k=1}^2 \|x-\xi_k\|_2^2- \|\xi_{k}^{i_k}-\xi_k\|_2^2\Big\}\le \|\bm x\|^2-\max_{x\in\mathcal X}\|\bm x\|^2\le 0.
\end{align*}
Thus, the corresponding constraint and epigraphical variable can be removed. 

\bibliography{alias,references} 

\begin{thebibliography}{10}
\providecommand{\url}[1]{#1}
\csname url@samestyle\endcsname
\providecommand{\newblock}{\relax}
\providecommand{\bibinfo}[2]{#2}
\providecommand{\BIBentrySTDinterwordspacing}{\spaceskip=0pt\relax}
\providecommand{\BIBentryALTinterwordstretchfactor}{4}
\providecommand{\BIBentryALTinterwordspacing}{\spaceskip=\fontdimen2\font plus
\BIBentryALTinterwordstretchfactor\fontdimen3\font minus
  \fontdimen4\font\relax}
\providecommand{\BIBforeignlanguage}[2]{{%
\expandafter\ifx\csname l@#1\endcsname\relax
\typeout{** WARNING: IEEEtranS.bst: No hyphenation pattern has been}%
\typeout{** loaded for the language `#1'. Using the pattern for}%
\typeout{** the default language instead.}%
\else
\language=\csname l@#1\endcsname
\fi
#2}}
\providecommand{\BIBdecl}{\relax}
\BIBdecl

\bibitem{LA-MF-JL-FD:23}
L.~Aolaritei, M.~Fochesato, J.~Lygeros, and F.~Dörfler, ``Wasserstein tube
  {MPC} with exact uncertainty propagation,'' \emph{arXiv preprint
  arXiv:2304.12093}, 2023.

\bibitem{LA-NL-HC-FD:23}
L.~Aolaritei, N.~Lanzetti, H.~Chen, and F.~Dörfler, ``Distributional
  uncertainty propagation via optimal transport,'' \emph{arXiv preprint
  arXiv:2205.00343}, 2023.

\bibitem{DB-MK-TL-AP-SE:22}
D.~Bartl, M.~Kupper, T.~Lux, A.~Papapantoleon, and S.~Eckstein, ``Marginal and
  dependence uncertainty: Bounds, optimal transport, and sharpness,''
  \emph{SIAM Journal on Control and Optimization}, vol.~60, no.~1, pp.
  410--434, 2022.

\bibitem{AB:02}
A.~Barvinok, \emph{A course in convexity}.\hskip 1em plus 0.5em minus
  0.4em\relax American Mathematical Society, 2002, vol.~54.

\bibitem{DB-SES:96}
D.~Bertsekas and S.~E. Shreve, \emph{Stochastic optimal control: the
  discrete-time case}.\hskip 1em plus 0.5em minus 0.4em\relax Athena
  Scientific, 1996.

\bibitem{DB-DBB-CC:11}
D.~Bertsimas, D.~B. Brown, and C.~Caramanis, ``Theory and applications of
  robust optimization,'' \emph{SIAM Review}, vol.~53, no.~3, p. 464–501,
  2011.

\bibitem{PB:08}
P.~Billingsley, \emph{Probability and measure}.\hskip 1em plus 0.5em minus
  0.4em\relax John Wiley, 2008.

\bibitem{JB-YK-KM:16}
J.~Blanchet, Y.~Kang, and K.~Murthy, ``Robust {W}asserstein profile inference
  and applications to machine learning,'' \emph{Journal of Applied
  Probability}, vol.~56, no.~3, pp. 830--857, 2019.

\bibitem{JB-YK-KM-FZ:19}
J.~Blanchet, Y.~Kang, K.~Murthy, and F.~Zhang, ``Data-driven optimal transport
  cost selection for distributionally robust optimization,'' in \emph{2019
  Winter Simulation Conference (WSC)}, 2019, pp. 3740--3751.

\bibitem{JB-KM:19}
J.~Blanchet and K.~Murthy, ``Quantifying distributional model risk via optimal
  transport,'' \emph{Mathematics of Operations Research}, vol.~44, no.~2, pp.
  565--600, 2019.

\bibitem{JB-KM-NS:21}
J.~Blanchet, K.~Murthy, and N.~Si, ``Confidence regions in {W}asserstein
  distributionally robust estimation,'' \emph{Biometrika}, vol. 109, pp.
  295--–315, 2021.

\bibitem{JB-KM-FZ:22}
J.~Blanchet, K.~Murthy, and F.~Zhang, ``Optimal transport-based
  distributionally robust optimization: Structural properties and iterative
  schemes,'' \emph{Mathematics of Operations Research}, vol.~47, no.~2, pp.
  1500--1529, 2022.

\bibitem{DB-JC-SM:21-tac}
D.~Boskos, J.~Cort\'es, and S.~Martinez, ``Data-driven ambiguity sets with
  probabilistic guarantees for dynamic processes,'' \emph{IEEE Transactions on
  Automatic Control}, vol.~66, no.~7, pp. 2991--3006, 2021.

\bibitem{DB-JC-SM:23}
D.~Boskos, J.~Cort{\'e}s, and S.~Mart{\'i}nez, ``High-confidence data-driven
  ambiguity sets for time-varying linear systems,'' \emph{IEEE Transactions on
  Automatic Control}, pp. 1--16, 2023.

\bibitem{SPB-LV:04}
S.~P. Boyd and L.~Vandenberghe, \emph{Convex optimization}.\hskip 1em plus
  0.5em minus 0.4em\relax Cambridge University Press, 2004.

\bibitem{HB:10}
H.~Brezis, \emph{Functional Analysis, Sobolev Spaces and Partial Differential
  Equations}.\hskip 1em plus 0.5em minus 0.4em\relax Springer, 2010.

\bibitem{GCC-LEG:06}
G.~C. Calafiore and L.~E. Ghaoui, ``On distributionally robust
  chance-constrained linear programs,'' \emph{Journal of Optimization Theory \&
  Applications}, vol. 130, no.~1, pp. 1--22, 2006.

\bibitem{LMC-DB-TO:22}
L.~M. Chaouach, D.~Boskos, and T.~Oomen, ``Uncertain uncertainty in data-driven
  stochastic optimization: towards structured ambiguity sets,'' in \emph{{IEEE}
  Int. Conf. on Decision and Control}, 2022, pp. 4776--4781.

\bibitem{RC-ICP:18}
R.~Chen and I.~C. Paschalidis, ``A robust learning approach for regression
  models based on distributionally robust optimization,'' \emph{Journal of
  Machine Learning Research}, vol.~19, no.~13, pp. 1--48, 2018.

\bibitem{PC-PP:22}
P.~Coppens and P.~Patrinos, ``Data-driven distributionally robust {MPC} for
  constrained stochastic systems,'' \emph{IEEE Control Systems Letters},
  vol.~6, pp. 1274--1279, 2022.

\bibitem{JC-JL-FD:21}
J.~Coulson, J.~Lygeros, and F.~Dörfler, ``Distributionally robust chance
  constrained data-enabled predictive control,'' \emph{IEEE Transactions on
  Automatic Control}, vol.~67, no.~7, pp. 3289--3304, 2022.

\bibitem{JD-FM:19}
J.~Dedecker and F.~Merlev{\`e}de, ``Behavior of the empirical {W}asserstein
  distance in ${R}^d$ under moment conditions,'' \emph{Electronic Journal of
  Probability}, vol.~24, 2019.

\bibitem{ED-YY:10}
E.~Delage and Y.~Ye, ``Distributionally robust optimization under moment
  uncertainty with application to data-driven problems,'' \emph{Operations
  Research}, vol.~58, no.~3, p. 595–612, 2010.

\bibitem{SD-MS-RS:13}
S.~Dereich, M.~Scheutzow, and R.~Schottstedt, ``Constructive quantization:
  {A}pproximation by empirical measures,'' \emph{Annales de l'Institut Henri
  Poincaré, Probabilités et Statistiques}, vol.~49, no.~4, p. 1183–1203,
  2013.

\bibitem{TE-BF-MH-RN:15}
T.~Eisner, B.~Farkas, M.~Haase, and R.~Nagel, \emph{Operator theoretic aspects
  of ergodic theory}.\hskip 1em plus 0.5em minus 0.4em\relax Springer, 2015.

\bibitem{PME-DK:17}
P.~M. Esfahani and D.~Kuhn, ``Data-driven distributionally robust optimization
  using the {W}asserstein metric: performance guarantees and tractable
  reformulations,'' \emph{Mathematical Programming}, vol. 171, no. 1-2, pp.
  115--166, 2018.

\bibitem{KF:53}
K.~Fan, ``Minimax theorems*,'' \emph{Proceedings of the National Academy of
  Sciences}, vol.~39, no.~1, pp. 42--47, 1953.

\bibitem{NF-AG:15}
N.~Fournier and A.~Guillin, ``On the rate of convergence in {W}asserstein
  distance of the empirical measure,'' \emph{Probability Theory and Related
  Fields}, vol. 162, no. 3-4, p. 707–738, 2015.

\bibitem{RG:22}
R.~Gao, ``Finite-sample guarantees for {W}asserstein distributionally robust
  optimization: Breaking the curse of dimensionality,'' \emph{Operations
  Research}, vol.~0, no.~0, p. null, 0.

\bibitem{RG-AJK:17}
R.~Gao and A.~J. Kleywegt, ``Data-driven robust optimization with known
  marginal distributions,'' 2017.

\bibitem{RG-AJK:23}
------, ``Distributionally robust stochastic optimization with {W}asserstein
  distance,'' \emph{Mathematics of Operations Research}, vol.~48, no.~2, pp.
  603--655, 2023.

\bibitem{IG-DB-LL-MM:23}
I.~Gracia, D.~Boskos, L.~Laurenti, and {M. Mazo Jr.}, ``Distributionally robust
  strategy synthesis for switched stochastic systems,'' in \emph{Proceedings of
  the 26th ACM International Conference on Hybrid Systems: Computation and
  Control}, 2023, pp. 1--10.

\bibitem{AH-IY:21}
A.~Hakobyan and I.~Yang, ``Wasserstein distributionally robust motion control
  for collision avoidance using conditional value-at-risk,'' \emph{IEEE
  Transactions on Robotics}, vol.~38, no.~2, pp. 939--957, 2021.

\bibitem{RJ-YG:16}
R.~Jiang and Y.~Guan, ``Data-driven chance constrained stochastic program,''
  \emph{Mathematical Programming}, vol. 158, no. 1-2, p. 291–327, 2016.

\bibitem{AK:13}
A.~Klenke, \emph{Probability theory: a comprehensive course}.\hskip 1em plus
  0.5em minus 0.4em\relax Springer, 2013.

\bibitem{BK:12}
B.~Kloeckner, ``Approximation by finitely supported measures,'' \emph{ESAIM:
  Control, Optimisation and Calculus of Variations}, vol.~18, pp. 343--359,
  2012.

\bibitem{DK-PME-VAN-SAS:19}
D.~Kuhn, P.~M. Esfahani, V.~A. Nguyen, and S.~Shafieezadeh-Abadeh,
  ``Wasserstein distributionally robust optimization: {T}heory and applications
  in machine learning,'' in \emph{Operations research \& management science in
  the age of analytics}.\hskip 1em plus 0.5em minus 0.4em\relax Informs, 2019,
  pp. 130--166.

\bibitem{BL-YT-AGW-GRD:22}
B.~Li, Y.~Tan, A.~Wuo, and G.~Duan, ``A distributionally robust optimization
  based method for stochastic model predictive control,'' \emph{IEEE
  Transactions on Automatic Control}, vol.~67, no.~11, pp. 5762--5776, 2022.

\bibitem{DL-DF-SM:19-ecc}
D.~Li, D.~Fooladivanda, and S.~Mart{\'\i}nez, ``Data-driven variable speed
  limit design for highways via distributionally robust optimization,'' in
  \emph{{E}uropean {C}ontrol {C}onference}, 2019, pp. 1055--1061.

\bibitem{KM:15}
K.~Marti, \emph{Stochastic Optimization Methods: Applications in Engineering
  and Operations Research}, 3rd~ed.\hskip 1em plus 0.5em minus 0.4em\relax
  Springer, 2015.

\bibitem{GP-DW:07}
G.~Pflug and D.~Wozabal, ``Ambiguity in portfolio selection,''
  \emph{Quantitative Finance}, vol.~7, no.~4, pp. 435--442, 2007.

\bibitem{BKP-ARH-SB-DSC-AC:20}
B.~K. Poolla, A.~R. Hota, S.~Bolognani, D.~S. Callaway, and A.~Cherukuri,
  ``Wasserstein distributionally robust look-ahead economic dispatch,''
  \emph{IEEE Transactions on Power Systems}, vol.~36, no.~3, pp. 2010--2022,
  2020.

\bibitem{IP:07}
I.~Popescu, ``Robust mean-covariance solutions for stochastic optimization,''
  \emph{Operations Research}, vol.~55, no.~1, pp. 98--112, 2007.

\bibitem{SS-DK-PME:19}
S.~Shafieezadeh-Abadeh, D.~Kuhn, and P.~M. Esfahani, ``Regularization via mass
  transportation,'' \emph{Journal of Machine Learning Research}, vol.~20, no.
  103, pp. 1--68, 2019.

\bibitem{SSA-VAN-DK-PME:18}
S.~Shafieezadeh-Abadeh, V.~A. Nguyen, D.~Kuhn, and P.~M. Esfahani,
  ``Wasserstein distributionally robust {K}alman filtering,'' in \emph{Advances
  in Neural Information Processing Systems}, 2018, pp. 8474--8483.

\bibitem{AS-DD-AR:14}
A.~Shapiro, D.~Dentcheva, and A.~Ruszczyński, \emph{Lectures on Stochastic
  Programming: Modeling and Theory}.\hskip 1em plus 0.5em minus 0.4em\relax
  SIAM, 2014, vol.~16.

\bibitem{NS-JB-SG-MS:20}
N.~Si, J.~Blanchet, S.~Ghosh, and M.~Squillante, ``Quantifying the empirical
  wasserstein distance to a set of measures: Beating the curse of
  dimensionality,'' in \emph{Advances in Neural Information Processing
  Systems}, 2020, pp. 21\,260--21\,270.

\bibitem{IT-CDC-TC:15}
I.~Tzortzis, C.~D. Charalambous, and T.~Charalambous, ``Dynamic programming
  subject to total variation distance ambiguity,'' \emph{SIAM Journal on
  Control and Optimization}, vol.~53, no.~4, pp. 2040--2075, 2015.

\bibitem{IT-CDC-CNH:21}
I.~Tzortzis, C.~D. Charalambous, and C.~N. Hadjicostis, ``A distributionally
  robust {LQR} for systems with multiple uncertain players,'' in \emph{{IEEE}
  Int. Conf. on Decision and Control}, 2021, pp. 3972--3977.

\bibitem{BPGVP-DK-PJG-MM:15}
B.~P.~G. {Van Parys}, D.~Kuhn, P.~J. Goulart, and M.~Morari, ``Distributionally
  robust control of constrained stochastic systems,'' \emph{IEEE Transactions
  on Automatic Control}, vol.~61, no.~2, pp. 430--442, 2015.

\bibitem{CV:08}
C.~Villani, \emph{Optimal transport: old and new}.\hskip 1em plus 0.5em minus
  0.4em\relax Springer, 2008, vol. 338.

\bibitem{JW-FB:19}
J.~Weed and F.~Bach, ``Sharp asymptotic and finite-sample rates of convergence
  of empirical measures in {W}asserstein distance,'' \emph{Bernoulli}, vol.~25,
  no.~4A, pp. 2620--2648, 2019.

\bibitem{WW-DK-MS:14}
W.~Wiesemann, D.~Kuhn, and M.~Sim, ``Distributionally robust convex
  optimization,'' \emph{Operations Research}, vol.~62, pp. 1358--1376, 12 2014.

\bibitem{FW-MEV-BH:22}
F.~Wu, M.~E. Villanueva, and B.~Houska, ``Ambiguity tube {MPC},''
  \emph{Automatica}, vol. 146, p. 110648, 2022.

\bibitem{IY:21}
I.~Yang, ``Wasserstein distributionally robust stochastic control: A
  data-driven approach,'' \emph{IEEE Transactions on Automatic Control},
  vol.~66, no.~8, pp. 3863--3870, 2021.

\end{thebibliography}
\bibliographystyle{IEEEtranS}
\end{document}